\documentclass[11pt, a4paper]{amsart}

\usepackage{verbatim, latexsym, amssymb, amsmath,color}
\usepackage{epsfig}
\def\R{\mathbb R}
\def\N{\mathbb N}
\def\Z{\mathbb Z}
\def\C{\mathbb C}

\def\H{\mathcal H}

\def\x{\mathbf x}

\DeclareMathOperator\U{U}

\DeclareMathOperator\SU{SU}
\def\y{\mathbf y}

\def\vol{\mathrm{vol}}

\def\hess{\mathrm{Hess}\,}

\def\dist{\mathrm{dist}}

\newtheorem*{thma}{Theorem A}
\newtheorem*{thmb}{Theorem B}

\newtheorem{thm}{Theorem}[section]
\newtheorem{lemm}[thm]{Lemma}
\newtheorem{cor}[thm]{Corollary}
\newtheorem{prop}[thm]{Proposition}
\newtheorem{lem}[thm]{Lemma}
\theoremstyle{remark}
\newtheorem{rmk}[thm]{Remark}

\theoremstyle{definition}

\newtheorem{dfn}[thm]{Definition}

\def\eq#1{{\rm(\ref{#1})}}
\def\R{\mathbb{R}}
\def\C{\mathbb{C}}
\def\N{\mathbb{N}}
\def\bfx{{\bf x}}
\def\bfy{{\bf y}}

\def\bfv{{\bf v}}
\def\d{{\rm d}}
\def\w{\wedge}
\DeclareMathOperator\loc{loc}
\DeclareMathOperator\Div{div}
\DeclareMathOperator\id{id}

\title{Uniqueness of Lagrangian self-expanders}
\author{Jason D.~Lotay and Andr\'e Neves}
\address{Department of Mathematics \\ University College London \\ Gower Street \\ London WC1E 6BT \\ United Kingdom}
\email{j.lotay@ucl.ac.uk}
\address{Imperial College London\\ Huxley Building \\ 180 Queen's Gate \\ London SW7 2RH \\ United Kingdom}
\email{a.neves@imperial.ac.uk}
\thanks{The first author was supported by an EPSRC Career Acceleration Fellowship.  The second author was partly supported by Marie Curie IRG Grant and ERC Start Grant.}

\begin{document}

\maketitle

\begin{abstract} {We} show that zero-Maslov class Lagrangian self-expanders in $\C^n$ which are asymptotic to a pair of planes intersecting {transversely} are locally unique if $n>2$ and unique if $n=2$.
\end{abstract}

\section{Introduction}

Self-similar solutions to mean curvature flow model the flow behaviour near a singularity. If the initial condition for the flow is a zero-Maslov class Lagrangian in $\C^n$, it is well known \cite[Corollary 3.5]{neves2} that self-shrinkers are trivial (i.e., stationary solutions) and so the ones left to study are self-expanders. These are Lagrangians $L\subset \C^n$ so that $L_t=\sqrt{2t}L$ is a solution to mean curvature flow.

 Moreover, it is shown in \cite{neves4} that blow-downs of  eternal solutions to Lagrangian mean curvature flow (like translating solutions for instance) are self-expanders for positive time. Thus if one wants to understand whether or not non-trivial translating solutions can occur as blow-ups of finite time singularities of Lagrangian mean curvature flow, it is important that we understand self-expanders first.
 
 Another related perspective on self-expanders is that they are the simplest solutions to mean curvature flow which start on cones and hence could be seen as models to start the flow with singular initial condition.

The first examples of Lagrangian self-expanders were constructed in \cite{Anciaux,lee1,lee2}. In \cite{jlt} Joyce--Lee--Tsui generalized these constructions and in particular they  found, for any two Lagrangian planes $P_1,P_2\subset \C^n$ satisfying an angle criterion,   explicit examples of zero-Maslov {class} Lagrangians asymptotic to these planes. They are diffeomorphic to $S^{n-1}\times\R$ and can  be seen  as the equivalent of Lawlor necks for the self-expander equation.  {The construction in \cite{jlt}} is  quite general and {provides}  examples  which are asymptotic to  non-stationary cones and examples which have Maslov class. {Further examples were constructed in \cite{castro}.}

 Given a Lagrangian cone in $\C^n$ which is graphical over a real plane and such that the potential function has eigenvalues uniformly in $(-1,1)$,  Chau, Chen and He \cite{chau} showed there is a unique graphical Lagrangian self-expander  asymptotic to that cone. 

Let $P_1,P_2\subset \C^n$ be two Lagrangian planes intersecting transversely, denote the space of bounded smooth functions with compact support by $C_0^{\infty}(\C^n)$
 and let $\mathcal{H}^n$ be $n$-dimensional Hausdorff measure.

\begin{dfn}\label{asympdfn} We say the  self-expander $L$ is asymptotic to $L_0=P_1+P_2$ if
$$\lim_{t\to 0}\int_{\sqrt{2t}L}\phi \,\d\H^n=\int_{L_0}\phi \,\d\H^n
$$
for all $\phi\in C_0^{\infty}(\C^n)$. 
\end{dfn}

In this paper we first show local uniqueness.

\begin{thma}
	Assume that neither $P_1+P_2$ nor $P_1-P_2$ are area-minimizing.
	
	 Let  $L$  be  a smooth zero-Maslov class Lagrangian self-expander in $\C^n$ asymptotic to $P_1+P_2$. 
	
	There is $R_0>0$ and $\varepsilon>0$ so that any smooth zero-Maslov class {Lagrangian} self-expander which is
	\begin{itemize}
	\item asymptotic to $P_1+P_2$;
	\item $\varepsilon$-close in $C^2$ to $L$ in $B_{R_0}$;
	\end{itemize}
	coincides with $L$.
\end{thma}

{The idea to prove Theorem A is classical. We show that the linearization of the self-expander equation {defines a Banach space isomorphism} 
and then we apply {the} Inverse Function Theorem to obtain local uniqueness.}

When $n=2$ we improve this result and show  global uniqueness.

\begin{thmb}
	Assume that neither $P_1+P_2$ nor $P_1-P_2$ are area-minimizing.
	
	Smooth zero-Maslov class Lagrangian self-expanders asymptotic to $P_1+P_2$ are unique and thus coincide with one of the examples found by  Joyce--Lee--Tsui.
\end{thmb}

\begin{rmk}
It is known that special Lagrangians in $\C^2$ which are asymptotic to a pair of planes are unique modulo scaling and rigid motions. This uses the fact that, after a hyperk\"ahler rotation of  the complex structure, special Lagrangian surfaces become holomorphic curves. No similar characterization holds for Lagrangian self-expanders and hence the need for a different idea in Theorem B.

Moreover, without the smoothness assumption the uniqueness statement does not hold, as can be seen in \cite{nakahara}.
\end{rmk}

{We now briefly describe the idea behind the proof of Theorem B. 

The key result is to show Theorem \ref{compactness}, which says that the set of self-expanders in $\C^2$ which are asymptotic to a pair of transverse planes is compact. Assuming this result the idea, given a pair of planes $P_1,P_2$ as in Theorem B, {is to} deform $P_2$ into a plane $Q_2$ so that the Lagrangian angle remains constant and  $P_1,Q_2$ become {\em equivariant}, i.e., share the same $S^1$-symmetry.  From Theorem A we  can accompany  the deformation of the planes with a (local) deformation of {any} self-expander $L$ {asymptotic to $P_1+P_2$}. Theorem \ref{compactness} {ensures} that this local deformation can be carried all the way until we obtain a self-expander $Q$  asymptotic {to}  $P_1+Q_2$. {Since}  $P_1,Q_2$ are equivariant, it is simple  to show that $Q$ is unique (see Lemma \ref{anciaux}) and hence $L$ must have been unique as well.

Roughly speaking, the proof of Theorem \ref{compactness} rests on the fact that every non-trivial special Lagrangian cone in $\C^2$ has area-ratios not smaller than two, but the area-ratios of any self-expander as in Theorem B are strictly smaller than two, i.e.,  the area-ratios are too small for a singularity to develop.  
}

\vskip0.1in
{\bf Organization:} in Section \ref{basic} we introduce the basic concepts. 

In Section \ref{exponential.decay} we show that zero-Maslov class self-expanders asymptotic to a transverse intersection of planes have exponential decay outside a compact set.

In Section \ref{section.fredholm} we develop the Fredholm theory for the linearization of the self-expander equation.

In Section \ref{local.section} we show that zero-Maslov class self-expanders in $\C^n$ which are asymptotic to a transverse intersection of a non-area-minimizing pair of planes are locally unique. This implies Theorem A.

In Section \ref{compactness.section} we show that, given a compact set of {transversely} intersecting non-area-minimizing pairs of planes in $\C^2$, the family of zero-Maslov class self-expanders in $\C^2$ asymptotic to those pairs of planes is also compact.

In Section \ref{uniqueness.section} we use the work of the previous section and Section \ref{section.fredholm} to show global uniqueness for zero-Maslov class self-expanders in $\C^2$ which are asymptotic to a transverse intersection of a {non-}area-minimizing pair of planes. This proves Theorem B. 

\vskip0.1in
{\bf Acknowledgements:}  Both authors would like to thank Dominic Joyce for comments on an earlier version of this paper.

\section{Basic theory and notation}\label{basic}

Consider $\C^n$  endowed with its usual complex coordinates $z_j=x_j+iy_j$, for $j=1,\ldots,n$, complex structure $J$, 
 K\"ahler form $\omega=\sum_{j=1}^n\d x_j\w\d y_j$ and holomorphic volume form $\Omega=\d z_1\w\ldots\w\d z_n$.  
 Observe that the Liouville form $\lambda=\sum_{j=1}^n(x_j\d y_j-y_j\d x_j)$ satisfies $\d\lambda=2\omega$.

Let $L$ be a connected \emph{Lagrangian} in $\C^n$; 
that is, $L$ is a (real) $n$-dimensional submanifold of $\C^n$ such that $\omega|_L\equiv 0$.  
Let $\bfx$ denote the position vector on $L$, let $\nabla$ be the (induced) Levi-Civita connection on $L$ and let $H$ be the mean curvature vector on $L$.  Standard  Euclidean differentiation is denoted by $\overline\nabla$.

Notice that $\lambda$ is trivially a closed 1-form on $L$. 
We say that $L$ is \emph{exact} if there exists 
$\beta\in C^{\infty}(L)$ such that $\d\beta=\lambda|_L$.

Since $\Omega|_L$ is a unit complex multiple of the volume form at each point
 on $L$, we may define the \emph{Lagrangian  angle} $\theta$ on $L$ by the
  formula $\Omega|_L=e^{i\theta}\vol_L$.  We also have the relation $H=J\nabla\theta$
   (c.f.~\cite[Lemma 2.1]{ThomasYau}).
The Maslov class on $L$ is defined by the cohomology class of $\d\theta$, so $L$ 
 has \emph{zero-Maslov class} if $\theta$ is a single-valued function.

Observe that, since $T\C^n|_L=TL\oplus NL$, we may decompose any vector $\bfv$ on $L$ into tangential and normal components, denoted 
  $\bfv^\top$ and $\bfv^\bot$ respectively.

\begin{dfn}\label{selfexpdfn}  
  We say that $L$ is a \emph{self-expander} if $H=\kappa\bfx^{\bot}$ for
   some $\kappa>0$.  By rescaling $L$ we may assume that $\kappa=1$.  
 \end{dfn}
The importance of self-expanders $L$ with $H=\bfx^{\bot}$ is that $L_t=\sqrt{2t}L$ for $t>0$ solves mean curvature flow.  

We have the following basic properties of self-expanders.

\begin{lem}\label{selfexplem1}
\begin{itemize}\item[]
\item[(i)] Lagrangian self-expanders with zero-Maslov class are exact.
\item[(ii)] Let $L$ be a zero-Maslov class self-expander. Then $L$ 
is a self-expander with $H=\bfx^{\perp}$ 
if and only if $\beta+\theta$ is constant.
\end{itemize}
\end{lem}
\begin{proof}Let $L$ satisfy $H=\bfx^{\perp}$.  
Since $H=J\nabla\theta$, one sees that $$\nabla\theta=-J(\bfx^{\bot})=-(J\bfx)^\top,$$ so $\lambda|_L=-\d\theta$. This proves the first property.

To prove the second property note that 
$$H-\bfx^\bot =0\iff \nabla \theta+J \bfx^\bot=0\iff \nabla \theta+(J \bfx)^\top=0\iff\nabla(\theta+\beta)=0.$$
\end{proof}

 Let $P_1$, $P_2$ be two Lagrangian planes intersecting transversely. From \cite[Section 7.2]{ilmanen}, there exists a constant $C_0$  such that whenever  {a} self-expander $L$ is asymptotic to $P_1+P_2$ then
\begin{equation}\label{area.bound0}
\H^n\big(L\cap B_R\big)\leq C_0R^n\quad\mbox{for all } R>0,
\end{equation}
where $B_R$ will always denote $B_R(0)$, the ball of radius $R$ about $0$ in $\C^n$.

In this paper, all self-expanders $L$ we consider have the following properties:
\begin{itemize}
\item $L$ is Lagrangian with zero-Maslov class;
\item $L$ has  $H={\bf x}^{\bot}$;
\item $L$ is asymptotic  to $L_0=P_1+P_2$, where $P_1,P_2$ are {transversely} intersecting Lagrangian planes.
\end{itemize}
 We abuse notation and often identity the varifold $L_0=P_1+P_2$ with its support $L_0=P_1\cup P_2$.

A 
key tool in studying self-expanders is the backwards {heat kernel}.
\begin{dfn}\label{densitydfn}  
Given any $(x_0,l)$ in $\C^{n}\times \R$, we consider the backwards heat kernel
\begin{equation}\label{heat}
\Phi(x_0,l)(x,t)=\frac{\exp\left(-\frac{|x-x_0|^2}{4(l-t)}\right)}{(4\pi(l-t))^{n/2}}\,.
\end{equation}
\end{dfn}

Given a solution $(L_t)_{t> 0}$ to mean curvature flow and $x_0\in\C^n$, $l>0$, we consider
 \begin{equation}\label{gaussian}
 \Theta_t(x_0,l)=\int_{L_t}\Phi(x_0,l) \d\H^n.
 \end{equation} 
 Note that when $L_t=\sqrt{2t}L$, where $L$ is a self-expander, we have that $\Theta_t(x_0,l)$  is finite due to \eqref{area.bound0} (see \cite[Lemma C.3]{eckernotes}).  
 
 Definition \ref{asympdfn} implies that for all $x_0\in\C^n$ and $l>0$
 \begin{equation}\label{gaussian.zero}
\lim_{t\to 0} \Theta_t(x_0,l)=\int_{L_0} \Phi(x_0,l) \d\H^n=\Theta_0(x_0,l).
 \end{equation}
Moreover, we have from Huisken's monotonicity formula \cite{Huisken} that
\begin{equation}\label{Thetadec}
\Theta_t(x_0,l)\leq \Theta_0(x_0,l+t)\quad\mbox{for all }x_0\in\C^n,\,t> 0,\, l>0.
\end{equation}

We conclude this section with the following observation. Given $P_1, P_2$ transverse planes we have
\begin{equation}\label{lessthantwo}
\Theta_{0}(x_0,l)=\int_{P_1+P_2}\Phi(x_0,l) \d\H^n<2\quad\mbox{for all }l>0
\end{equation}
unless $x_0=0$. One consequence of this observation  is the following.
\begin{lem}\label{embedlem}
\begin{itemize}\item[]
\item The self-expander $L$ is embedded.
\item There is $c_1$ depending only on $L_0$ 
 so that
 $$ \H^n\big(L\cap B_r(x)\big)\leq c_1r^n\quad\mbox{for all }r>0\mbox{ and }x\in\C^n.$$
\end{itemize}
\end{lem}
\begin{proof}
Suppose that $L$ is immersed.  Then there exists $x_0\in L$ where
$$\lim_{\delta\to 0}\Theta_{\frac 1 2}(x_0,\delta)\geq 2.$$
 By \eq{Thetadec}, making $t=\frac 1 2$ and $l\to 0$, we obtain 
 $$ 2\leq \Theta_0(x_0,1/2)$$
 and so, by  \eqref{lessthantwo}, $x_0=0$ and equality holds in the equation above. In this case,
 $$2\leq \lim_{\delta\to 0}\Theta_{\frac 1 2}(0,\delta)\leq \Theta_0(x_0,1/2)=2$$
 and so equality holds in Huisken's monotonicity formula.  Hence $L$ is also a self-shrinker (i.e. $H=-\bfx^{\perp}$)  and  thus $H=\bfx^{\perp}=0$ as $L$ is a self-expander. Therefore, $L$ must be a
 cone but, because $L$ is asymptotic to $L_0$, this is only possible if $L=L_0$, which then contradicts the assumption that $L$ is smooth. This proves the first property.
 
  In what follows $c$ denotes a constant depending only on  $n$.   By \eq{Thetadec} and \eqref{lessthantwo}
 $$  \frac{\H^n\big(L\cap B_r(x)\big)}{r^n}\leq c\Theta_{1/2}(x,r^2)\leq c\Theta_0(x,r^2+1/2)\leq 2c.$$
\end{proof}
\section{Exponential decay}\label{exponential.decay}

In this section we show that the self-expander $L$ converges exponentially fast to $L_0$ outside a compact set.  {This naturally coincides with the behaviour of the relevant self-expanders in \cite{jlt}, 
but is in marked contrast to special Lagrangian Lawlor necks which only converge at rate $O(r^{1-n})$ to their asymptotic planes.}

Let $G_L(n,\C^n)$ denote the set of all Lagrangian planes in $\C^n$. Consider the open subset of $G_L(n,\C^n)\times G_L(n,\C^n)$ given by
$$G_n=\{(P_1,P_2)\in G_L(n,\C^n)\times G_L(n,\C^n)\, |\, P_1\cap P_2=\{0\}\}.$$ 
Given a compact set $K\subset G_n$, we denote by $\mathcal{S}(K)$ the set of all self-expanders which are asymptotic to $L_0=P_1+P_2$, with $(P_1,P_2)\in K$.

\begin{thm}\label{thm.asympt}
For every compact set $K\subset G_n$ and $k\in\N$, there is $R_1>0$, $C$ and $b$ so that for all $L\in{\mathcal S}(K)$ we find $\psi \in C^{\infty}(L_0\setminus B_{R_1})$ satisfying
$$L\setminus B_{2R_1}\subset\{x+J\overline\nabla \psi(x)\,|\, x\in L_0\setminus B_{R_1}\}  \subset L\setminus B_{R_1/2}$$
and 
$$\|\psi\|_{C^{k}(L_0\setminus B_R)}\leq Ce^{-bR^2}\mbox{ for all }R\geq R_1.$$
\end{thm} 
\begin{proof}

The next proposition says that if $L$ is locally graphical over $L_0\setminus B_{R_0}$  for some $R_0$ and the local graph is asymptotic to zero in the $C^{k+1}$-norm, then we can find $R_1$ large so that $L$ is a global graph over $L_0\setminus B_{R_1}$ and the graph has its $C^k$-norm decaying exponentially fast.
\begin{prop}\label{graphdecay} 
Fix $R_0>2r_0>0$,  $k\in\N$, {a compact set $K\subset G_n$, and a decreasing function ${D_k=D_k(r)}$ tending to zero at infinity.} 

Suppose that for every $L\in{\mathcal S}(K)$ and  $y_0\in  L\setminus B_{R_0}$ we can find $x_0\in {L_0}$ and 
$$\phi:L_0\cap B_{2r_0}(x_0)\rightarrow \C^n$$ so that
\begin{itemize}
\item the $C^{k+1}\big(B_{2r_0}(x_0)\big)$-norm of $\phi$ is bounded by ${D}_k(|y_0|)$;\smallskip
\item $L\cap \hat B_{r_0}(y_0)\subset \{x+J\overline\nabla\phi(x)\,|\,x\in L_0\cap B_{2r_0}(x_0)\}$, {where $\hat B_{r_0}(y_0)$ denotes the intrinsic ball of $L$ with radius $r_0$ and centered at $y_0$}.
\end{itemize}
Then there exist $R_1$, $C,$ $b,$ and an open set $B\subset\C^n$ with compact closure, depending on $r_0,$ $R_0,$ ${D}_k$ and $K$, such that for every $L \in{\mathcal S}(K)$ we can find $\psi\in C^{\infty}(L_0\setminus B_{R_1})$ with 
\begin{equation}\label{graphical}
L\setminus B\subset\{x+J\overline\nabla\psi(x)\,| \,x\in L_0\setminus B_{R_1}\}
\end{equation}
and 
\begin{equation}\label{expdecay}
\|\psi\|_{C^{k}(L_0\setminus B_R)}\leq Ce^{-bR^2}\quad\mbox{for all }R\geq R_1.
\end{equation}
\end{prop}

\begin{proof}
From the hypotheses of the proposition, for every $L\in {\mathcal S}(K)$  and ${\tilde R_1}$ sufficiently large we can find an open set ${\tilde B}\subset \C^n$ with compact closure and a projection map
$$\pi_L:L\setminus{ \tilde B} \longrightarrow L_0\setminus  B_{{\tilde R_1}}.$$
We claim  we can choose $R_1=R_1(r_0,R_0,C_k)$  so that $\pi_L$ is a diffeomorphism {when restricted to $\pi_L^{-1}({L_0\setminus}B_{R_1})$}.

Suppose not. Then we can find $L^i\in  {\mathcal S}(K)$ with $L^i_0=P^i_1+P^i_2$ tending to $P_1+P_2$, $x_i$ in $L^i_0\setminus B_{R_0}$ such that $|x_i|\rightarrow\infty$ as $i\rightarrow\infty$, and $\pi_{L^i}^{-1}(x_i)\supset\{y_i,z_i\}$ where $y_i\neq z_i$ for all $i$.

  By hypothesis, there exists $\delta_0>0$ such that  $L^i\cap B_{2\delta_0}(x_i)$ contains the graphs of functions $f_i,g_i$ on 
$B_{\delta_0}(x_i)\cap L^i_0$ with $f_i(x_i)=y_i$ and $g_i(x_i)=z_i$ for all $i$.  Therefore, recalling $\Phi$ given in \eq{heat}, Huisken's monotonicity formula implies
\begin{align}\label{multone}
\int_{L_0^i}\Phi(x_i,\delta^2+1/2) \d\mathcal{H}^n & \geq \int_{L^i}\Phi(x_i,\delta^2) \d\mathcal{H}^n\\ \notag
 & \geq \int_{L^i\cap B_{2\delta_0}(x_i)}\Phi(x_i,\delta^2) \d\mathcal{H}^n\\ \notag
&\geq \int_{\text{graph}(f_i)\cup\text{graph}(g_i)}\Phi(x_i,\delta^2) \d\mathcal{H}^n\\ \notag
&\geq 2\int_{B_{\delta_0}(x_i)\cap L^i_0}\Phi(x_i,\delta^2) \d\mathcal{H}^n, \notag
\end{align}
where the last inequality comes from the fact that  $L^i_0$ is a union of planes. 

{Since} $L^i_0$ tends to $P_1+P_2$ with $P_1\cap P_2=\{0\}$ and $|x_i|\rightarrow\infty$, we {have from the hypothesis of the proposition that $L^i_0-x_i$ tends to either $P_1$ or $P_2$. We assume that the first case occurs and so, } 
after translating by $-x_i$,
 $B_{\delta_0}(x_i)\cap L^i_0$ tends to $B_{\delta_0}(0)\cap  P_1$.  Thus, letting $i\rightarrow\infty$ and then $\delta\rightarrow 0$, we obtain from \eqref
{multone} that 
$$ 1=\int_{P_1}\Phi(0,1/2) \d\mathcal{H}^n\geq 2.$$
This proves the claim.

{The fact that} $\pi_L$ is a diffeomorphism implies the existence of a smooth vector field $X$ on $L_0\setminus B_{R_1}$ so that, {for some open set $B$},
$$L\setminus B=\{x+JX(x)\,:\,x\in L_0\setminus B_{R_1}\}.$$

Since $L$ is Lagrangian, it is standard to see that $X=\overline\nabla \psi$ locally.  We need to make sure that $\psi$ is defined globally
on $L_0\setminus B_{R_1}$.  Recall the primitive $\beta$ for the Liouville form $\lambda$ on $L$. 
Set $\bar{\beta}=\beta \circ\pi_L^{-1}$ and 
define $\psi$ on $L_0\setminus B_{R_1}$ by
$$\psi(x)=\frac{1}{2}\left(\langle X(x), x\rangle-\bar \beta(x)\right),$$
using the standard Euclidean inner product $\langle\,,\,\rangle$.
Then,  for every vector $v$  in $L_0$ we have
\begin{align*}
        \langle \overline \nabla \psi(x), v\rangle & =\textstyle\frac{1}{2}\left(\langle \overline\nabla_v X, x\rangle +\langle X(x), v\rangle-\langle \overline\nabla \bar \beta(x), v\rangle \right)\\
        & =\textstyle\frac{1}{2}\big(\langle \overline\nabla_v X, x\rangle +\langle X(x), v\rangle-\lambda(v+J\overline\nabla_v X(x))\big)\\
        &=\textstyle\frac{1}{2}\left(\langle \overline\nabla_v X, x\rangle +\langle X(x), v\rangle-\langle Jx-X(x), v+J\overline\nabla_v X(x)\rangle\right)\\        
&=\langle X(x), v\rangle-\textstyle\frac{1}{2}\langle Jx,v\rangle +\textstyle\frac{1}{2}\langle X(x),J\overline\nabla_v X(x)\rangle.
\end{align*}
 Observe that 
$\langle Jx,v\rangle=0$ since $L_0$ is a {pair of} Lagrangian
 planes and thus $x$ is tangent to $L_0$ and orthogonal to $Jv$.    
Furthermore, $\langle X(x),J\overline\nabla_vX(x)\rangle=0$ because $\overline\nabla_v X$ has no 
component orthogonal to $L_0$ as $L_0$ is a union of planes.  
Thus $\overline\nabla\psi=X$, the decomposition \eq{graphical} holds, and for $R\geq R_1$,
\begin{equation}\label{psi.decay}
\|\nabla \psi\|_{C^{k+1}(L_0\setminus B_{R})}\leq {D}_k(R-1),
\end{equation}
where ${D}_k$ is our given decreasing function by hypothesis.

We now wish to prove \eq{expdecay}. {    Let $S_0$ be a connected component of $L_0\setminus\{0\}$,{ i.e., $S_0$ is a Lagrangian plane minus the origin,} and let $S$ be the connected component of $L\setminus B_{R_1}$ asymptotic to
$S_0$ given by the graphical decomposition in \eq{graphical}.  After changing coordinates so that $S_0$ is identified with a real plane, we 
 consider the following vector{-}valued function defined on {$S$:}}
$$\bfy ={(J\overline \nabla \psi)}\circ\pi_L={i\sum_{j=1}^ny_j \frac{\partial}{\partial y_j}\in \C^{n}}.$$
Let $\Delta=-\d^*\d$ be the analyst's Laplacian acting on functions on $L$. {For} $j=1,\ldots,n$ we have on $S$ that
$$\Delta y_j^2=2{y_j\left\langle i \frac{\partial}{\partial y_j},H \right\rangle}+2\left| \partial_ {y_j}^{\top}\right|^2\implies \Delta |\y|^2\geq 2\left\langle \y,H \right\rangle.$$
 Calculating with respect to the induced metric on $S$  and recalling that $\bfx$ is the position vector on $L$  we see that, since $L$ is a self-expander,
$$ |\y|^2=\langle \y,\x \rangle=\langle \y,\x^{\bot} \rangle+\langle \y,\x^{\top} \rangle=\langle\y,H\rangle+\langle \y,\x^{\top} \rangle=\langle\y,H\rangle+\frac{1}{2}\langle \nabla |\y|^2,\x \rangle.$$
 Thus, if we define
\begin{equation}\label{Leq}
\mathcal{L}(\phi)=\Delta\phi+\langle\bfx,\nabla\phi\rangle-2\phi
\end{equation}
for suitably differentiable functions $\phi$ on $S$, where all quantities are computed with respect to the induced metric on $S$,
 we see that $$\mathcal{L}(|\bfy|^2)\geq 0.$$
 We are now in the position to construct a barrier for $|\bfy|^2$ and  deduce \eq{expdecay}.

Set $\rho(x)=\exp(-|x|^2/2)$, an ambient function on $\C^n$.  Since $L$ is a self-expander, we have that 
$$\Delta({|x|}^2)=2n+2\langle H,\bfx\rangle=2n+2|\bfx^\bot|^2.$$
Therefore one sees that, on $L$,
$$\Delta\rho=-\frac\rho 2 \Delta({|x|}^2)+\frac\rho 4|\nabla{|x|}^2|^2=\rho(|\bfx^\top|^2-n-|\bfx^\bot|^2).$$
Moreover,
$$\langle\nabla \rho, \x\rangle =-\rho|\x^{\top}|^2$$
and so
\begin{equation}\label{rhoeq} 
\mathcal L(\rho)=
\rho(|\x^{\top}|^2-n-2-|\x|^2)\leq -(n+2)\rho.
\end{equation}

Let $\varepsilon>0$ and using \eqref{psi.decay} choose ${C}={C}(r_0,R_0,{D}_k)$ so that $$|\y|^2={|(\overline \nabla \psi)\circ\pi_L|}^2<{C}\exp(-{|x|}^2/2)\mbox{ on }\partial S.$$
Set $\tilde{\rho}=\varepsilon+{C}\rho$. 
For all $R$ sufficiently large we have $|\y|^2<\tilde \rho$ on $S\cap \partial \{R_1<|x|<R\}$ because $|\y|^2$ tends to zero at 
infinity by hypothesis.  Furthermore, {using} $\mathcal{L}(|\bfy|^2)\geq 0$ and \eqref{rhoeq}, we have 
$$\mathcal L(\tilde \rho-|\y|^2)=\mathcal L(\rho-|\y|^2)-2\varepsilon< 0$$
and thus the Maximum Principle implies that $$|\y|^2<\varepsilon+{C}\exp(-{|x|}^2/2)\mbox{ on }S\cap \{R_1<|x|<R\}.$$ 
Letting $R\rightarrow\infty$ and then $\varepsilon\rightarrow 0$, we conclude that $$|\y|^2\leq {C}\exp(-{|x|}^2/2)\mbox{ on }S.$$

Recall that on {$S$} we have $\bfy={(J\overline\nabla\psi)\circ \pi_L}$.  Therefore we can add a constant to $\psi$  and find some other constant ${C}={C}(r_0,R_0,{D}_k)$ so that, after integration, 
$$|\psi(x)|\leq {C}\exp\left(-\frac{1}{4}|x|^2\right)\quad\mbox{for all}\quad x\in S_0\setminus B_{R_1}.$$  
As the $C^{k+1}$ norm of $\psi$  in 
$S_0\setminus B_{R_1}$ is bounded it follows from standard interpolation inequalities for H\"older spaces (see e.g.~\cite[Theorem 3.2.1]{Krylov}) that, for some {further} constant $C=C(r_0,R_0,{D}_k)${,} $$\|\psi\|_{C^{k}(L_0\setminus B_{R})}\leq C\exp(-aR^2)$$
for some constant $a>0$ and any $R\geq R_1$.

 We can argue in the same manner for each connected component of $L_0\setminus B_{R_1}$ 
and conclude the desired result.
\end{proof}

We now make an observation concerning $\Theta$ given in \eq{gaussian}. Given $L^i\in{\mathcal S}(K)$, we denote by $\Theta^i_t(x_0,l)$ the Gaussian density ratios \eqref{gaussian} evaluated at $L^i_t$.

\begin{lemm}\label{lemm.asympt.zero}
Let $L^i$ be a sequence in ${\mathcal S}(K)$ and $(x_i)_{i\in \N}$ a sequence of points in $\C^n$ with $|x_i|$ tending to infinity.
 Then, for all $l>0$, $$ \lim_{i\to\infty}\Theta^i_0(x_i,l)\leq 1$$ with equality only if $\lim_{i\to\infty}\dist(x_i, L^i_0)=0$. %
\end{lemm}
\begin{proof}
We have $L^i_0=P^i_1+P^i_2$ and write $x_i=a^1_i+b^1_i=a^2_i+b^2_i$, where $a^j_i \in P^i_j$ and $\langle a^j_i,b^j_i\rangle=0$ for $j=1,2$.  We set $Q_i=L^i_0-
x_i
$, where we mean that we translate $L^i_0$ by the vector given by $x_i$. We have
\begin{equation}\label{dist.plane}
\min\{|b^1_i|,|b^2_i|\}=\dist(x_i, {L^i_0})=\dist(0,Q_i)
\end{equation}

Suppose first that
$\limsup_{i\to\infty}|b^1_i|=\limsup_{i\to\infty}|b^2_i|=+\infty$.
We then have $\dist(0,Q_i)$ tending to infinity and so
$$\lim_{i\to\infty}\Theta_0(x_i,l)= \lim_{i\to\infty} \int_{Q_i}\Phi(0,l) \d\H^n =0.$$

Otherwise, without loss of generality, 
$\limsup_{i\to\infty}|b^1_i|<\infty$
and necessarily $\lim\inf_{i\to\infty}|a^1_i|=+\infty$. Note that we must also have $\lim\inf_{i\to\infty}|b^2_i|=+\infty$  because otherwise we could extract a subsequence of $(P^i_1,P^i_2)$ converging to a pair of planes intersecting along a line. 
 
Therefore $Q_i$ sequentially converges to $P_1+b$, i.e. an affine plane parallel to some plane $P_1$, where $b=\lim_{i\to\infty}b^1_i$ is orthogonal to $P_1$. Thus 
\begin{align*}
\lim_{i\to\infty}\Theta^i_0(x_i,l)&=\int_{P_1+b}\Phi(0,l) \d\H^n= \int_{P_1}\Phi(-b,l) \d\H^n\\
&=\exp\left(-\frac{|b|^2}{4l}\right)\int_{P_1}\Phi(0,l) \d\H^n=\exp\left(-\frac{|b|^2}{4l}\right)\leq 1,
\end{align*}
with equality only if $b=0$. This proves the desired result.
\end{proof}

We can now {finish the proof of Theorem \ref{thm.asympt}}. The idea is to show that the hypotheses of Proposition \ref{graphdecay} are satisfied for all $L\in {\mathcal S}(K)$.

{First we claim that 
\begin{equation}\label{distance.zero}
\lim_{R\to\infty} \sup\{\dist(x, L_0)\,|\, L \in {\mathcal S}(K),  x\in L\setminus B_R \}=0.
\end{equation}
Indeed, if we choose any sequence $L^i\in {\mathcal S}(K)$ and pick $x_i\in L^i$ with $|x_i|$ tending to infinity, we have from \eq{Thetadec} that
$$1=\lim_{r\to 0}\Theta^i_{1/2}(x_i,r)\leq \lim_{r\to 0} \Theta^i_{0}(x_i,r+1/2)=\Theta^i_0(x_i,1/2)$$
and thus 
$$\lim_{i\to\infty}\Theta^i_0(x_i,1/2)\geq 1.$$ 
The claim now follows from Lemma \ref{lemm.asympt.zero}.}

{Second we claim the existence of $R_1>0$ so that the $C^{2,\alpha}$ norm of  $L\setminus B_{R_1}$ is uniformly bounded  for all $L\in \mathcal S(K)$. Indeed, with $\varepsilon_0>0$ fixed, we obtain from Lemma \ref{lemm.asympt.zero} the existence of $R_1$ so that  for all $L\in\mathcal{S}(K)$ we have
$$\Theta_0(x,2)\leq 1+\varepsilon_0\mbox{ for all }|x|\geq R_1/2$$
and so, for all $t,l\in[0,1]$ we have
$$ \Theta_t(x, l)\leq \Theta_0(x,l+t)\leq \Theta_0(x,2)\leq 1+\varepsilon_0 \mbox{ for all }|x|\geq R_1/2.$$
White's Regularity Theorem \cite[Theorem 3.1]{white} implies the desired claim.

From the first and second claim we see that given $r>0$ and $\varepsilon>0$ we can find $R_2$ so that for all $L\in \mathcal{S}(K)$ and $x\in L\setminus B_{R_2}$ we have that $\hat B_r(x)\cap L$  is $\varepsilon$-close in $C^{2,\alpha}$ to a ball of radius $r$ in $L_0$. Elliptic regularity implies that for every $k\in \N$ we can choose  $R_2$ larger so that  $\hat B_r(x)\cap L$  is $\varepsilon$-close in $C^{k+1,\alpha}$ to a ball of radius $r$ in $L_0$.  Thus, for every $k\in \N$, we can find $r_0,R_0$, and  $D_k$ so that the  hypotheses of Proposition \ref{graphdecay} are satisfied for all $L\in \mathcal{S}(K). $
}
This implies the desired result.
\end{proof}

\section{Fredholm Theory}\label{section.fredholm}

In this section we develop the Fredholm theory for the operator 
 $$\mathcal{L}(\phi)=\Delta\phi+\langle\bfx,\nabla\phi\rangle-2\phi$$ 
defined on the self-expander 
$L$, which already arose in \eq{Leq}. The relevant spaces to consider are given below.

\begin{dfn}\label{Hk*dfn}
For $k\in\Z^+$, let $H^k(L)$ denote the Sobolev space $W^{k,2}(L)$ with norm
$$\|\phi\|_{H^k}=\left(\sum_{j=0}^k{\int_L}|\nabla^j\phi|^2 \d\mathcal{H}^n\right)^{\frac{1}{2}}.$$
 Let $H^k_*(L)$ denote the subspace of $H^k(L)$ such that the norm
$$\|\phi\|_{H^k_*}=\left(\sum_{j=0}^k\int_L|\nabla^j\phi|^2 \d\mathcal{H}^n+\sum_{j=1}^{k-1}\int_L\langle\bfx^{\top},\nabla^j\phi\rangle^2
 \d\mathcal{H}^n\right)^{\frac{1}{2}}$$
is finite.  Both $H^k$ and $H^k_*$ are Banach spaces (in fact, Hilbert spaces).
\end{dfn}

Our first result, which shows the utility of Definition \ref{Hk*dfn}, is the following.

\begin{prop}\label{Lctsprop}
The map $\mathcal{L}:H^{k+2}_*(L)\rightarrow H^k(L)$ is well-defined and continuous.
\end{prop}

\begin{proof}
 From Theorem \ref{thm.asympt} we have that every derivative of the second fundamental form of $L$ is uniformly bounded. 
Thus if $T$ is a tensor on $L$ we have from the Bochner formula and Gauss equation that
\begin{equation}\label{bochner.fredholm}
\nabla \Delta T=\Delta \nabla T+C_1\star \nabla T+C_0\star T,
\end{equation}
where $C_0, C_1$ are two uniformly bounded tensors (depending only on $L$) and $A\star B$ denotes any contraction of tensor $A$ with tensor $B$. Moreover, using the fact that $H=\x^{\bot}$ on $L$ we also have
\begin{equation}\label{bochner2.fredholm}
\nabla \langle \x^{\top}, T\rangle= \langle \x^{\top},\nabla T \rangle+C_0\star T,
\end{equation}
 where $C_0$ is another uniformly bounded tensor which depends only on $L$.
 
 We can use \eqref{bochner.fredholm} and \eqref{bochner2.fredholm} inductively to conclude that for all $i\in \N$ and $\phi\in C_0^{\infty}({L})$ we have
 \begin{equation}\label{induct.fredholm}
 \nabla^i\mathcal{L}(\phi)=\Delta \nabla^{i}\phi+ \langle \x^{\top}, \nabla^{i+1}\phi\rangle+\sum_{j=0}^i C_j\star \nabla^j\phi,
 \end{equation}
 where $C_j$ are uniformly bounded tensors. Thus for all $i\in \N$ we can find a positive constant $c_i=c_i(L)$ so that 
\begin{equation*}
\int_{L}|\nabla^i\mathcal{L}(\phi)|^2\d\H^n\leq c_i\left(\sum_{j=0}^{i+2} \int_{L}|\nabla^j\phi|^2 \d\H^n+  \int_{L}\langle \x^{\top}, \nabla^{i+1}\phi\rangle^2 \d\H^n\right).
\end{equation*}
 It follows that $\mathcal{L}$ defines a continuous map from $H^{k+2}_*(L)$ to $H^k(L)$.  
\end{proof}

\begin{thm}\label{main.fredholm} The map $\mathcal {L}: H^2_{*}(L)\longrightarrow L^2(L)$ is an isomorphism.
\end{thm}
\begin{proof}
We start by proving the existence of a positive constant $C_0=C_0(L,n)$ so that for all $\phi$ in $H^2_{*}(L)$ we have
\begin{equation}\label{inj.fredholm}
\|\phi\|^2_{H^2_{*}}\leq C_0\int_{L} |\mathcal {L}(\phi)|^2\d\H^n=C_0\|\mathcal {L}(\phi)\|^2_{L^2}.
\end{equation}
Using the fact that $H=\x^{\bot}$ on $L$, we have
\begin{equation}\label{div}
{\rm div}\, \bfx^{{\top}}=n+|\bfx^{{\bot}}|^2
\end{equation}
and so direct computation shows that 
\begin{equation}\label{grad.fredholm}
2\phi \langle \x,\nabla \phi\rangle=\Div(\x^{\top}\phi^2)-(n+|\x^{\bot}|^2)\phi^2.
\end{equation}
Thus
$$-\int_{L}\mathcal{L}(\phi)\phi\, \d\H^n=\int_{L}|\nabla\phi|^2+(2+n/2+|\x^{\bot}|^2/2) \phi^2\d\H^n$$
and so, since $-\mathcal{L}(\phi)\phi\leq |\mathcal{L}(\phi)|^2+\frac{1}{4}\phi^2$, we obtain
\begin{equation}\label{cauchy.fredholm} 
\int_{L}|\nabla\phi|^2 \d\H^n\leq  \int_{L}|\mathcal{L}(\phi)|^2 \d\H^n.
\end{equation} 
Moreover, if $A$ is the second fundamental form of $L$, 
\begin{multline}\label{div.fredholm}
\Delta\phi\langle \x,\nabla \phi\rangle =\Div\left(\nabla \phi\langle \x^{\top},\nabla \phi\rangle-\x^{\top}\frac{|\nabla \phi|^2}{2}\right)+\frac{|\nabla \phi|^2}{2}(n-2+|\x^{\bot}|^2)\\
-\langle H, A(\nabla \phi, \nabla \phi)\rangle
\end{multline}
and combining  \eqref{div.fredholm} with \eqref{grad.fredholm} we obtain
\begin{multline}\label{l2.fredholm}
\int_{L} |\mathcal {L}(\phi)|^2\d\H^n=\int_L (\Delta \phi)^2+\langle \x,\nabla \phi\rangle^2+(n+2+|\x^{\bot}|^2)|\nabla \phi|^2\d\H^n\\
+2\int_L(n+2+|\x^{\bot}|^2)\phi^2\d\H^n-2\int_L \langle H, A(\nabla \phi, \nabla \phi)\rangle \d\H^n.
\end{multline}
From the Bochner formula and Gauss equation  there is a constant $c=c(n)$ so that on $L$ we have $$|\nabla \Delta\phi-\Delta \nabla \phi|\leq c|A|^2|\nabla \phi|$$
and thus
\begin{equation}\label{hess.fredholm}
|\nabla^2 \phi|^2\leq (\Delta \phi)^2+\Div\big(\nabla^2\phi(\nabla\phi,\cdot)-\Delta\phi\nabla \phi\big)+c|A|^2|\nabla \phi|^2.
\end{equation}
Inserting  \eqref{hess.fredholm} in \eqref{l2.fredholm} we have
\begin{multline}\label{l2l2.fredholm}
\int_L|\nabla^2 \phi|^2+\langle \x,\nabla \phi\rangle^2+|\nabla \phi|^2+\phi^2 \d\H^n\leq \int_{L} |\mathcal {L}(\phi)|^2\d\H^n\\
+(c+2)\int_L |A|^2|\nabla \phi|^2 \d\H^n.
\end{multline}
From Theorem \ref{thm.asympt} we know that $|A|^2$ is uniformly bounded on $L$ and so we have from \eqref{cauchy.fredholm} the existence of a constant $C=C(L,n)$ so that
$$
\int_L |A|^2|\nabla \phi|^2 \d\H^n\leq  C\int_{L}|\mathcal{L}(\phi)|^2 \d\H^n.
$$
This last inequality and \eqref{l2l2.fredholm} imply \eqref{inj.fredholm} at once.

The immediate consequence of \eqref{inj.fredholm} is that $\mathcal{L}$ is injective. It also follows that its range is closed for 
the following reason. If $v_i=\mathcal{L}(\phi_i)$ is a sequence converging in $L^2(L)$ to $v$, then \eqref{inj.fredholm} implies  
$(\phi_i)_{i\in \N}$ is a Cauchy sequence in  $H^2_{*}(L)$ and thus sequentially converging to $\phi \in H^2_{*}(L)$. Naturally, 
$\mathcal{L}(\phi)=v$ since $\mathcal{L}$ is continuous by Proposition \ref{Lctsprop}.

We now argue that $\mathcal{L}$ is surjective. In order to do so we compute the formal adjoint $\mathcal{L}^*$ of  $\mathcal{L}$.
 {Using \eq{div} we have,} for every $u,v\in C^{\infty}(L)$, 
\begin{equation}\label{uv.fredholm}
\langle \x,\nabla u\rangle v=\Div(\x^{\top}uv)-(n+|\x^{\bot}|^2)uv- \langle \x,\nabla v\rangle u.
\end{equation}
Hence, if $\phi\in C^{\infty}_0(L)$ and $\eta \in C^{\infty}(L)$,
$$
\int_{L}\mathcal{L}(\phi)\eta \,\d\mathcal{H}^n=\int_L\phi\big(\Delta\eta-\langle\bfx,\nabla\eta\rangle-(n+|\bfx^\bot|^2+2)\eta\big) \d\mathcal{H}^n
$$
which means 
\begin{equation}\label{adj.fredholm}
\mathcal{L}^*(\eta)=\Delta\eta-\langle\bfx,\nabla\eta\rangle-(n+|\bfx^\bot|^2+2)\eta.
\end{equation}

Suppose that $\mathcal{L}$ is not surjective. Since its range is closed we can find $\eta \in L^2(L)$ non-zero such that
$$\int_L \mathcal{L}(\phi)\eta\, \d\H^n=0\mbox{ for all }\phi \in C^{\infty}_0(L).$$
Elliptic regularity implies that $\eta$ is smooth and hence a solution to $\mathcal{L}^*(\eta)=0$. The next lemma implies that $\eta=0$ which is a contradiction.

\begin{lemm}\label{injectprop}  If $\eta\in C^{\infty}(L)\cap L^2(L)$ and $\mathcal{L}^*(\eta)=0$, then $\eta=0$.
\end{lemm}
\begin{proof}
We start by arguing that $\eta\in H^1(L)$.  For each $R>0$ consider a cut-off function $\phi_R\in C^{\infty}_0(\C^n)$ so that
\begin{equation}\label{phi.conds}
0\leq \phi_R\leq1 ,\quad \phi_R|_{B_R}=1,\quad \mbox{supp}(\phi_R)\subseteq B_{2R}, \quad|\nabla\phi_R|\leq c_0R^{-1},
\end{equation}
for some universal constant $c_0$.

Using \eqref{uv.fredholm} with $u=\eta^2/2$ and $v=\phi_R^2$ we have
\begin{multline*}
0=-\int_{L}\mathcal{L}^*(\eta)\eta\phi_R^2\,\d\H^n=\int_{L}|\nabla \eta|^2\phi_R^2+\left(\frac{n+|\x^{\bot}|^2}{2}+2\right)\eta^2\phi_R^2\,\d\H^n\\
+\int_{L}\langle\nabla \eta, \nabla \phi_R^2\rangle\eta \d\H^n-\int_{L}\langle\x, \nabla \phi_R^2\rangle\frac{\eta^2}{2}\, \d\H^n.
\end{multline*}
{We have for some universal constant $c$ and all $\varepsilon>0$ that 
$$\int_{L}\langle\nabla \eta, \nabla \phi_R^2\rangle\eta \d\H^n\leq {c\varepsilon}\int_{L}|\nabla \eta|^2\phi_R^2\, \d\H^n+\frac{c}{\varepsilon} \int_L\eta^2\d\H^n$$
 and thus we find another}
 uniform constant $c$ so that
$$\int_{L}|\nabla \eta|^2\phi_R^2\, \d\H^n\leq c\int_{L}\langle\x, \nabla \phi_R^2\rangle\frac{\eta^2}{2} \,\d\H^n+c\int_L\eta^2\d\H^n.$$
The term $\langle\x, \nabla \phi_R^2\rangle$ is uniformly bounded (independent of $R$) by \eq{phi.conds} and so
$$ \int_{L}|\nabla \eta|^2\phi_R^2 \,\d\H^n\leq c\int_L\eta^2\d\H^n$$
for some other universal constant $c$. Letting $R$ tend to infinity we obtain that  $\eta\in H^1(L)$.

We now show that $\eta=0$.  Set 
$$f(r)=\int_{L\cap B_r}\eta^2 \d\mathcal{H}^n+\left(2+\frac{n}{2}\right)^{-1}\int_{L\cap B_r}|\nabla\eta|^2 \d\mathcal{H}^n.$$
The idea is to show that 
\begin{equation}\label{ode.fredholm}
\left(2+\frac{n}{2}\right)f(r)\leq rf^{\prime}(r){\mbox{ for almost all }r\geq 2+\frac n 2}.
\end{equation}
If true we obtain from integrating \eqref{ode.fredholm} that, for all $r\geq r_1{\geq 2+\frac n 2}$,
$$f(r)\geq f(r_1)\left(\frac{r}{r_1}\right)^{2+\frac{n}{2}}.$$
As $\eta \in H^1(L)$ the function $f$ is uniformly bounded, which contradicts the inequality above unless $f\equiv 0$, which means
 $\eta=0$.

Hence to complete the proof we need to show \eqref{ode.fredholm}. Applying integration by parts to the identity
$$\int_{L\cap B_r}\mathcal{L}^*(\eta)\eta\, \d\H^n=0$$
 and using \eqref{grad.fredholm} with $\phi=\eta$ we obtain
\begin{multline*}
\int_{L\cap B_r}|\nabla \eta|^2 +\left(2+\frac n 2+\frac{|\bfx^\bot|^2}{2}\right)\eta^2 \d\mathcal{H}^n\\
=-\frac{1}{2}\oint_{\partial(L\cap B_r)}\big(\langle\bfx^\top,\nu\rangle\eta^2-\partial_\nu\eta^2\big)\d\mathcal{H}^{n-1}.
\end{multline*}
Therefore
\begin{multline}\label{ff.fredholm}
\left(2+\frac{n}{2}\right)f(r)\leq \frac{r}{2}\oint_{\partial(L\cap B_r)}\eta^2 \d\mathcal{H}^{n-1}
+\oint_{\partial(L\cap B_r)}|\eta||\nabla\eta| \d\mathcal{H}^{n-1}\\
 \leq \frac{r+1}{2}\oint_{\partial(L\cap B_r)}\eta^2 \d\mathcal{H}^{n-1}+\frac{1}{2}\oint_{\partial(L\cap B_r)}
|\nabla\eta|^2 \d\mathcal{H}^{n-1}.
\end{multline}
On the other hand using the co-area formula we have, for almost all $r$, 
\begin{align*}
f^{\prime}(r)&=\oint_{\partial (L\cap B_r)} \eta^2\frac{|\bfx|}{|\bfx^\top|} \,\d\mathcal{H}^{n-1}+\left(2+\frac{n}{2}\right)^{-1}
\oint_{\partial (L\cap B_r)} |\nabla\eta|^2\frac{|\bfx|}{|\bfx^\top|} \,\d\mathcal{H}^{n-1}\\
&\geq \oint_{\partial (L\cap B_r)}\eta^2 \d\mathcal{H}^{n-1}+\left(2+\frac{n}{2}\right)^{-1}\oint_{\partial (L\cap B_r)}
|\nabla\eta|^2 \d\mathcal{H}^{n-1}.
\end{align*}
Combining this inequality with \eqref{ff.fredholm} we obtain  {that \eqref{ode.fredholm} holds for almost all $r\geq 2+\frac n 2$.}
 \end{proof}
  Applying Lemma \ref{injectprop} shows that $\mathcal{L}$ is surjective, completing the proof of Theorem \ref{main.fredholm}.\end{proof}

\begin{cor}\label{Lisocor} The map $\mathcal {L}: H^{k+2}_{*}(L)\longrightarrow H^k(L)$ is an isomorphism  for all $k\in \N$.
\end{cor}
\begin{proof}

We proceed by induction where the case $k=0$ follows from Theorem \ref{main.fredholm}. Assume Corollary \ref{Lisocor} holds for some
 $k\in \N$.  Thus $\mathcal{L}$ is injective and given $v\in H^{k+1}(L)\subset H^k(L)$ there is 
$\phi \in H^{k+2}_{*}(L)$ so that $\mathcal{L}(\phi)=v$.  We need to show that $\phi \in H^{k+3}_{*}(L)$ in order to prove the corollary.

Set $T=\nabla^{k+1}\phi$. From \eqref{induct.fredholm} there exists a constant $C_1=C_1(L,n,k)$ so that 
\begin{equation}\label{e2.fredholm}
\|\Delta T+\langle \x^{\top}, \nabla T\rangle\|^2_{L^2}\leq C_1\left(\|\phi\|^2_{H^{{k+1}}_{*}} +\|\mathcal {L}(\phi)\|^2_{H^{k+1}}\right).
\end{equation}
Next we argue that,  for some constant $C_2=C_2(L,n,k)$, 
\begin{equation}\label{l2norm.T}
\left|\int_{L}\langle\Delta T,\nabla_{ \x^{\top}}T\rangle \d\mathcal{H}^{n}\right|\leq C_2\int_{L}|\nabla T|^2+|T|^2\d\mathcal{H}^{n}.
\end{equation}
Reasoning as in \eqref{div.fredholm}{,} there is a universal constant $c=c(n,k)$ so that
\begin{multline*}
\left|\langle\Delta T,\nabla_{ \x^{\top}}T\rangle-\Div\left(\langle\nabla  T, \nabla_{\x^{\top}}T\rangle - \x^{\top}\frac{|\nabla T|^2}{2}\right)\right|\\
\leq c|\nabla T|^2(1+|A|^2+|\x^{\bot}|^2)+c|T||\nabla T||A|^2|\x^{\top}|.
\end{multline*}
From Theorem \ref{thm.asympt} we have that $|A|^2$ and $|H|^2=|\bfx^{\bot}|^2$ have exponential decay, so $|A|^2|\x|$ and $|A|^2+|\x^{\bot}|^2$ are uniformly bounded on $L$.  
Thus we can find a constant $C_2=C_2(L,n,k)$ so that
\begin{equation*}
\left|\langle\Delta T,\nabla_{ \x^{\top}}T\rangle-\Div\left(\langle\nabla  T, \nabla_{\x^{\top}}T\rangle - \x^{\top}\frac{|\nabla T|^2}{2}\right)\right|
\leq C_2(|\nabla T|^2+|T|^2),
\end{equation*}
which implies, after integration, inequality \eqref{l2norm.T}.

Combining \eqref{e2.fredholm} with \eqref{l2norm.T}, we have 
\begin{equation}\label{e1.fredholm}
\|\Delta T\|^2_{L^2}+\|\langle \x^{\top}, \nabla T\rangle\|^2_{L^2}\leq 
C_3\left(\|\phi\|^2_{H^{k+2}_{*}} +\|\mathcal {L}(\phi)\|^2_{H^{k+1}}\right)
\end{equation}
for some $C_3=C_3(C_1,C_2)$. Using the Bochner formula and the fact that $|A|^2$ is uniformly bounded on $L$ we can find a constant $C_4=C_4(L,n,k)$ so that
$$\|\nabla^2T\|^2_{L^2}\leq \|\Delta T\|^2_{L^2}+C_4\|\nabla T\|^2_{L^2}.$$
Combining this inequality with \eqref{e1.fredholm} we obtain
$$\|\phi\|^2_{H^{k+3}_{*}}\leq C_5\left(\|\phi\|^2_{H^{k+2}_{*}}+\|\mathcal {L}(\phi)\|^2_{H^{k+1}}\right)$$
for some constant $C_5$, which implies that $\phi \in H^{k+3}_{*}(L)$ as required. \end{proof}

\section{Local uniqueness}\label{local.section}

In this section we prove the local uniqueness of zero-Maslov class Lagrangian self-expanders $L$ asymptotic to 
transverse pairs of multiplicity one planes $L_0=P_1+P_2$.  

Recall the definition of $G_n$ in Section \ref{exponential.decay} and consider a smooth path $(P_1(s),P_2(s))\in G_n$ with $P_1(0)=P_1$ and $P_2(0)=P_2$. We assume that the difference of the Lagrangian angles $\theta\big(P_1(s)\big)-\theta\big(P_2(s)\big)$ is constant.
\begin{thm}\label{local.uniqueness} Given  a zero-Maslov class self-expander $L$ asymptotic to $L_0=P_1+P_2$, there is $R_0>0$, $s_0>0$ and $\varepsilon>0$ so that zero-Maslov class  self-expanders $L^s$  which satisfy
\begin{itemize}
\item $L^s$ is asymptotic to $L^s_0=P_1(s)+P_2(s)$ for some  $|s|\leq s_0$,
\item $L^s$ is $\varepsilon$-close in $C^2$ to $L$ in $B_{R_0}$,
\end{itemize}
exist and are unique. The  family $(L^s)_{|s|\leq s_0}$ is continuous in $C^{2,\alpha}$.
\end{thm}

We achieve this by studying the deformation theory of $L$ and applying the Implicit Function Theorem. 

 We start by constructing the tubular neighbourhoods of $L_0$ and $L$ that we require and derive some basic properties.

\subsection*{Symplectic preliminaries}  

The cotangent bundle of a Lagrangian $N$ has a natural symplectic structure which is exact, meaning that there is a tautological one form  $\tau\in \Lambda^1(T^*N)$ so that, if $\omega_N$ is the tautological symplectic  form on $T^*N$, then $\d\tau=-\omega_N$. The form $\tau$ is determined by the following property: if $\Xi\in \Lambda^1(N)$ and we consider the natural map $\Xi:N\rightarrow T^*N$, then $\Xi^*{\tau}=\Xi$.  We remark that on $\R^{2n}=T^*\R^n$, $\tau=\sum_{i=1}^ny_i\d x_i$ and hence $\tau$ is different from the Liouville form $\lambda$.

For $\Xi \in \Lambda^1(N)$, we let $\Gamma_{\Xi}$ denote the section of $T^*N$ given by $x\mapsto \big(x,\Xi(x)\big)$.

A symplectomorphism 
$\Phi:(M_1,\d\lambda_1)\rightarrow (M_2,\d\lambda_2)$ between exact symplectic manifolds is called exact if 
 $\Phi^*(\lambda_2)-\lambda_1$ is exact.  

In particular, in the case of $\R^n\subseteq\C^n$, the map from $T^*\R^n$ to $\C^n$ given by
\begin{equation}\label{sympleq}
\big(x=(x_1,\ldots,x_n),\sum_{j=1}^ny_j(x)\d x_j\big)\mapsto (x_1+iy_1,\ldots,x_n+iy_n)
\end{equation}
 is an exact symplectomorphism identifying the zero section with the real $\R^n$ in $\C^n$.

The construction of the tubular neighbourhood of $L_0=P_1\cup P_2$ is elementary.  Without loss of generality we may assume that $P_1$ is the real $\R^n\subseteq\C^n$ and that 
$P_2=A\cdot\R^n$ where $A=\text{diag}(e^{i\theta_1},\ldots,e^{i\theta_n})$.
Naturally, we may define symplectomorphisms $\Psi_j:T^*P_j\rightarrow \C^n$ for $j=1,2$, where $\Psi_1$ is given by \eq{sympleq}  and $\Psi_2=A\circ\Psi_1$. 
Clearly, there exists $\zeta>0$ so that if 
\begin{equation}\label{Vjeq}
V_j=\big\{\big(x,\xi(x)\big)\in T^*(P_j\setminus \{0\}):|\xi(x)|<2\zeta|x|\big\}
\end{equation}
then $\Psi_1(V_1)\cap\Psi_2(V_2)=\emptyset$.  This choice ensures that we can allow for rotations of $P_1$ and $P_2$ in a tubular neighbourhood which is symplectomorphic to an open neighbourhood in 
$T^*(L_0\setminus\{0\})$, so we have the following.

\begin{lemm}\label{conenbdlem}
Set
\begin{itemize}
\item $V_0=V_1\cup V_2\subseteq T^*(L_0\setminus\{0\})$;\smallskip
\item $T_0=\Psi_1(V_1)\cup\Psi_2(V_2)\subseteq\C^n;$\smallskip
\item $\Psi_0:V_0\rightarrow T_0$ defined by $\Psi_0|_{V_j}=\Psi_j$.
\end{itemize}  
Then $V_0$ and $T_0$ are open tubular neighbourhoods of $L_0\setminus\{0\}$ and  $\Psi_0$ is a symplectomorphism preserving the Liouville form.
 Moreover, any small rotation of either of the planes $P_1$ or $P_2$ remains in $T_0$. 
 \end{lemm}
 We now use $V_0$, $T_0$, and $\Psi_0$ to construct our tubular neighbourhoods of $L$.  The point will be to ensure that the
 symplectomorphism we construct is compatible with the standard symplectomorphism 
\eq{sympleq} over the planes. 

 By Theorem \ref{thm.asympt}, there is $R_1>0$ and an open set $B\subset \C^n$ with compact closure so that $$L\setminus B=\{x+J\overline\nabla\psi(x):x\in L_0\setminus B_{R_1}\}$$
for some $\psi\in C^{\infty}(L_0\setminus B_{R_1})$ whose $C^{2,\alpha}$ norm decays exponentially.  In particular, we may assume by making $R_1$  larger if necessary that 
$|\d\psi(x)|<\zeta|x|$ for all $x\in L_0\setminus B_{R_1}$.  Let 
\begin{equation}\label{pi.eq}
\pi:L_0\setminus B_{R_1}\rightarrow L\setminus B, \quad\pi(x)=x+J\overline\nabla\psi(x)
\end{equation}
 so that $\pi^{*}:T^*(L\setminus B)\rightarrow T^*(L_0\setminus B_{R_1})$ is an isomorphism.   

\begin{prop}\label{tubenbdthm}  Recall the notation of Lemma \ref{conenbdlem}.  
There exist 
\begin{itemize}
\item  open neighbourhoods $\hat V\subset V$ of $L$ in $T^*L$;
\item open 
 tubular neighbourhoods $\hat T\subset T$ of $L$ in $\C^n$;
 \item an exact symplectomorphism $\Psi:V\rightarrow T$ with $\Psi|_L=\id_L$ and $\hat T=\Psi(\hat V);$
 \end{itemize}
 such that 
\begin{equation}\label{Veq}
\pi^*(V)=\big\{\big(x,\xi(x)\big)\in T^*(L_0\setminus B_{R_1}):|\xi(x)|<\zeta|x|\big\}\subseteq V_0,
\end{equation}
$$\pi^*(\hat{V})=\big\{\big(x,\xi(x)\big)\in T^*(L_0\setminus B_{R_1}):|\xi(x)|<\frac{1}{2}\zeta|x|\big\},$$
 and
\begin{equation}\label{Psieq}
\Psi\circ{(\pi^*)}^{-1}\big(x,\xi(x)\big)=\Psi_0\big(x,\d\psi(x)+\xi(x)\big)\quad\text{for all }(x,\xi)\in\pi^*(V). 
\end{equation}
\end{prop}

\begin{proof}
Recall that $L$ is 
embedded by Lemma \ref{embedlem}.  

Since $L$ is Lagrangian, $T^*L\cong NL$ so we may consider the exponential map $\exp$ acting on $T^*L$.  
Given any compact $K\subset L$, we may apply the usual tubular neighbourhood theorem to give open neighbourhoods of $K$ in $T^*K$ and $\C^n$ 
 which are diffeomorphic via $\exp$.  Moreover, $\exp$ and its derivative act as the identity on $K$.  In particular, $\exp^*\lambda=\lambda_L$ on $K$, 
 where $\lambda_L$ is the Liouville form on $T^*L$.

We may define $V$ and $\Psi$ over $L\setminus B$ via \eq{Veq} and \eq{Psieq} so that $\Psi(V)$ is a tubular neighbourhood of 
$L\setminus B$. Using the exponential map over the remainder of $L$, we can extend to open neighbourhoods $V$, $T$ of $L$ in $T^*L$ and $\C^n$ and a
 diffeomorphism $\Phi:V\rightarrow T$ such that $\Phi$ and $\Psi$ agree over $L\setminus B$.  

As $\pi$ in \eq{pi.eq} is a diffeomorphism, the map $\pi^*:T^*(L\setminus B)\rightarrow T^*(L_0\setminus B_{R_1})$ is a symplectomorphism preserving the tautological 
1-form (see, for example, \cite[Theorem 2.1]{symplectic}).  Since $\Psi_0$ preserves the Liouville form, we see  that on $L\setminus B$ we have 
$\Phi^*\lambda-\lambda_L=\d\alpha_L$ for some smooth function $\alpha_L$.  We can smoothly cut-off the function $\alpha_L$ so that it is
defined on $T^*L$ and vanishes over a compact subset of $L$.  Hence $\Phi$ is an exact symplectomorphism outside some compact set.  The idea now is to 
essentially use Moser's trick to perturb $\Phi$ over a compact set to a global exact symplectomorphism. 

Define $$\lambda_t=(1-t)(\lambda_L+\d\alpha_L)+t\Phi^*\lambda,$$ so that $$\d\lambda_t=2(1-t)\omega_L+2t\Phi^*\omega$$ is a closed nondegenerate 2-form on $V$ for all 
$t\in[0,1]$.  Using nondegeneracy, we can uniquely solve 
$$X_t\lrcorner \d\lambda_t=\lambda_L+\d\alpha_L-\Phi^*\lambda$$
pointwise for $X_t$.  Since $\Phi^*\lambda=\lambda_L+\d\alpha_L$ over $L\setminus B$, we see that $X_t$ is zero outside a compact set, {so} we 
may solve for a smooth vector field $X_t$ on $V$ for all $t\in[0,1]$ by shrinking $V$ if necessary.  Moreover, 
$L$ (viewed as the zero section) is Lagrangian with respect to $\omega_L$ and $\Phi^*\omega$, {so we may deduce that} $\d\lambda_t|_L=0$ and {hence} $X_t|_L=0$ as well.   

Define diffeomorphisms $f_t$ on $V$ such that $f_0=\id$ and $\frac{\d}{\d t}f_t=X_t\circ f_t$.  Then
$$\frac{\d}{\d t}f_t^*\lambda_t=f_t^*\big(\Phi^*\lambda-\lambda_L-\d\alpha_L+\d(X_t\lrcorner\lambda_t)+X_t\lrcorner\d\lambda_t\big)
=\d f_t^*(X_t\lrcorner\lambda_t).$$
We deduce that $f_1^*\Phi^*\lambda-\lambda_L=f_1^*\lambda_1-f_0^*\lambda_0+\d\alpha_L$ is exact.  Moreover the diffeomorphism 
$f_1$ acts as the identity on $L$ and on $V$ over $L\setminus B$.  Hence we have an exact symplectomorphism $\Psi=\Phi\circ f_1:V\rightarrow T$ 
which satisfies \eq{Veq} and \eq{Psieq}.
The remainder of the proposition follows by taking appropriate open subsets of $V$ and $T$.
\end{proof}

Write $C^{2}_{\loc}(V)$ and $C^2_{\loc}(\hat{V})$ for the space of locally $C^2$ 1-forms $\Xi$ with graph $\Gamma_{\Xi}\subseteq V$ and $\Gamma_{\Xi}\subseteq \hat{V}$ respectively.  We use similar notation for $C_{\loc}^{\infty}(V)$.
For $\Xi\in C^2_{\loc}(V)$ we define a $C^2$-embedding $f^{\Xi}:L\rightarrow T$ by $$f^{\Xi}(x)=\Psi\big(x,\Xi(x)\big)$$ so that $f^{\Xi}(L)$ is the
 deformation of $L$ given by $\Xi$.     
 

We note that  $f^{\Xi}(L)$ is Lagrangian if and only if $\d\Xi=0$.  
 %
However, we want to restrict ourselves to exact zero-Maslov class deformations $f^{\Xi}(L)$ since, by Lemma \ref{selfexplem1}, we know that if $f^{\Xi}(L)$ is a self-expander it must be exact.  
This motivates the next lemma.

\begin{lemm}\label{exactlem}  
Let $L^{\prime}=f^{\Xi}(L)$ with $\Xi\in  C^2_{\loc}(V).$  Then $L^{\prime}$ is exact and 
zero-Maslov class if and only if $\Xi=\d\phi$ for some $\phi\in C_{\loc}^{3}(L)$ with $\Gamma_{\d\phi}\subseteq V$.

Moreover, if we set
\begin{itemize}
\item $\bar \phi$ defined on $T$ so that $\bar\phi\big(\Psi(x,\xi)\big)=\phi(x)$,
\item $H(L^s)$ the mean curvature of  $L^s=f^{s\d\phi}(L)$ and $\Delta^s$ the pullback to $L$ of the Laplacian on $L^s$,
\item $\theta^{\phi}$ the pullback to $L$ of the Lagrangian angle of $f^{\d\phi}(L)$, and
\item $\beta^{\phi}$ the pullback to $L$ of the primitive of the Liouville form of $f^{\d\phi}(L)$,
\end{itemize}
then  $\theta^{\phi}$ and $\beta^{\phi}$ can be given, respectively, by
$$\theta^{\phi}(x)=\theta(x)+\int_0^1\Delta^s \phi(x)-\langle H(L^s),\overline \nabla \bar \phi\rangle|_{f^{s\d\phi}(x)}\d s$$
and
$$\beta^{\phi}(x)=\beta(x)-2\phi(x)+\int_0^1\langle \bfx,\overline \nabla \bar \phi\rangle|_{f^{s\d\phi}(x)}\d s.$$

We also have
$$\left.\frac{\d\theta^{s\phi}}{\d s}\right|_{s=0}=\Delta\phi-\langle H,\overline \nabla \bar \phi\rangle$$
and
$$\left.\frac{\d\beta^{s\phi}}{\d s}\right|_{s=0}=-2\phi+\langle \bfx, \nabla\phi \rangle+\langle H,\overline \nabla \bar \phi\rangle.$$
\end{lemm}

\begin{proof}
We first show that  exactness of $f^{\Xi}(L)$ corresponds to exactness of $\Xi$. It suffices to see that $(f^{\Xi})^*(\sum_{i=1}^ny_i\d x_i)$ is exact because that differs from  $-(f^{\Xi})^*(\lambda/2)$ by an exact form.  Since $\Psi$ is an exact symplectomorphism 
there exists a function $\alpha_L$ on $T^*L$ such that $\Psi^*(\sum_{i=1}^ny_i\d x_i)=\tau+\d\alpha_L$.   By definition, $f^{\Xi}=\Psi\circ\Xi$ so 
  $$(f^{\Xi})^*\left(\sum_{i=1}^ny_i\d x_i\right)=\Xi^*\circ\Psi^*\left(\sum_{i=1}^ny_i\d x_i\right)=\Xi^*(\tau+\d\alpha_L)=\Xi+\d\Xi^*(\alpha_L).$$    
Hence $f^{\Xi}(L)$ is exact if and only if $\Xi$ is exact.  

We now compute the stated identities for $\beta^{\phi}$.
 Consider the vector field on $T^*L$ given by $$\tilde X|_{(x,\xi)}=\big(0,\d\phi(x)\big)\in T_{(x,\xi)}(T^*L),$$ the function $\tilde \phi$ on $T^*L$ given by $\tilde \phi(x,\xi)=\phi(x)$, and $X=\Psi_{*}(\tilde X)$ a vector field on $T$.  We have $\tilde X\lrcorner\Psi^*\omega=-\d\tilde \phi$ and so, because $\bar \phi\circ\Psi=\tilde \phi$, we obtain $X=J\overline \nabla \bar \phi$ on $T$. As a result we have
\begin{equation*}\label{defor.angle}
\frac{\d}{\d s}f^{s\d\phi}(x)=\Psi_{*}|_{(x,s\d\phi)}(\tilde X)=X|_{f^{s\d\phi}(x)}=J\overline \nabla\bar \phi|_{f^{s\d\phi}(x)}.
\end{equation*}

Therefore we obtain
 \begin{align*}
 \frac{\d}{\d s}(f^{s\d\phi})^{*}\lambda=(f^{s\d\phi})^{*}\mathcal{L}_{J\overline \nabla \bar\phi}\lambda &=(f^{s\d\phi})^{*}\left(\d(J\overline \nabla \bar\phi\lrcorner\lambda)+J\overline \nabla \bar\phi\lrcorner2\omega\right).\\
&=(f^{s\d\phi})^{*}\d(\langle\bfx,\overline \nabla \bar\phi\rangle)-2\d\phi.
\end{align*}
In light of this formula we see that if we define 
$$\beta^{\phi}(x)=\beta(x)-2\phi(x)+\int_0^1\langle \bfx,\overline \nabla \bar \phi\rangle|_{f^{s\d\phi}(x)}\d s,$$
we have $(f^{\d\phi})^{*}\lambda=\d\beta^{\phi}$ and, since $L$ is a self-expander,
 \begin{align*}\left.\frac{\d\beta^{s\phi}}{\d s}\right|_{s=0}=-2\phi+\langle \bfx, \overline \nabla\bar\phi \rangle
 &=-2\phi+\langle \bfx, \nabla\phi \rangle+\langle \bfx^{\bot}, \overline\nabla\bar\phi \rangle\\
 &=-2\phi+\langle \bfx, \nabla\phi \rangle+\langle H, \overline\nabla\bar\phi \rangle.
 \end{align*}
 
We now show that the zero-Maslov class condition imposes no condition on $\Xi=\d\phi$. We have $(f^{s\d\phi})^*(\Omega)=e^{i\theta^s}(f^{s\d\phi})^*\vol_{L^s}$ where $e^{i\theta^s}$ is an ${S}^1$-valued function on $L$. 
If the deformation vector $X$ were orthogonal to $L^s$, we would have by  \cite[Lemma 2.3]{ThomasYau} that
$$\frac{\d\theta^s}{\d s}=\Delta^s\phi.$$
The fact that $X$ might have a tangential component along $L^s$ implies
\begin{multline*}
\frac{\d\theta^s}{\d s}=\Delta^s\phi+\langle X,(f^{s\d\phi})_{*}\nabla\theta^s\rangle=\Delta^s\phi+\langle JX,H(L^s)\rangle\\
=\Delta^s\phi-\langle H(L^s),\overline \nabla \bar \phi\rangle.
\end{multline*}
Therefore, integrating this equation for $\theta^s$ together with the 
initial condition $\theta^0=\theta$ allows us to define a (single-valued) function on each $L^{s}$ which is the Lagrangian angle.  
 In particular, $f^{\d\phi}(L)$ has zero-Maslov class and we can set
 $$\theta^{\phi}(x)=\theta(x)+\int_0^1\Delta^s \phi(x)-\langle H(L^s),\overline \nabla \bar \phi\rangle|_{f^{s\d\phi}(x)}\d s.$$
The equation for $\frac{\d}{\d s}\theta^{s\phi}$ was computed above.
 \end{proof}

\subsection*{Proof of Theorem \ref{local.uniqueness}}

Consider smooth rotations $P_1(s)$ and $P_2(s)$ of the planes $P_1$ and $P_2$ respectively, so that the planes remain Lagrangian and the difference of their Lagrangian angles stays the same.
 In that case we  find a one parameter family of matrices $B_1(s), B_2(s)\in \U(n)$ with 
 \begin{itemize}
 \item $P_1(s)=B_1(s)\cdot P_1$, and $P_2(s)=B_2(s)\cdot P_2$,
 \item $B_1(s)B_2^{-1}(s)\in \SU(n)$ and $B_1(0)=B_2(0)=\id$,
 \item $\det B_1(s)=\det  B_2(s)=e^{i\theta(s)},$ where $\theta(s)$ is a smooth function with $\theta(0)=0$.
\end{itemize} 

Consider $G_s:\C^n\rightarrow \C^n$, a one-parameter group of diffeomorphisms with $G_0=\id$,
$$G_s(x)=B_1{(s)}(x)\mbox{ on }\Psi_1(V_1)\setminus B_{R_1}\!\quad\mbox{and}\quad\!  G_s(x)=B_2{(s)}(x)\mbox{ on } \Psi_2(V_2)\setminus B_{R_1}.$$

From Proposition \ref{tubenbdthm} we can find $s_0$ so that for all $|s|\leq s_0$ and $\d\phi\in C_{\loc}^3(\hat V)$ we have $G_s\circ f^{\d\phi}(L)\subset T$.
In this case we define the $C^2$ embedding
$$f^{\d\phi,s}:L\rightarrow T,\quad f^{\d\phi,s}(x)=G_s\circ\Psi\big(x,\d\phi(x)\big).$$

  Given $\phi\in C_{\loc}^3(L)$ so that $\d\phi\in C^2_{\loc}(\hat{V})$ and $|s|\leq s_0$, we set $L^{\phi,s}=f^{\d\phi,s}(L)$. Thus, by Lemma \ref{exactlem} we know that $L^{\phi,s}$ is an exact zero-Maslov class Lagrangian and we consider its Lagrangian angle $\theta^{\phi,s}$ and primitive for the Liouville form $\beta^{\phi,s}$ pulled-back to $L$, which are given by Lemma \ref{exactlem}. 

If $\phi$ has strong enough decay then $L^{\phi,s}$ is asymptotic to  $L^s_0=P_1(s)+P_2(s)$.
 For simplicity we write $L^{\phi}=L^{\phi,0}$, $\beta^{\phi}=\beta^{\phi,0}$ and $\theta^{\phi}=\theta^{\phi,0}$.

Lemma \ref{selfexplem1} shows that $L^{\phi,s}$ is a self-expander 
with $H=\bfx^{\bot}$ if and only if $\beta^{\phi,s}+\theta^{\phi,s}$ is constant. 
 This motivates our definition of a deformation map. 

\begin{dfn}\label{defmapdfn} We define a function $\hat{F}$ on functions $\phi\in C^3_{\loc}(L)$ such that $\d\phi\in C^2_{\loc}(\hat{V})$ and $|s|<s_0$ by 
$$\hat{F}(\phi,s)=\beta^{\phi,s}+\theta^{\phi,s}-\theta(s).$$  
We also let $F(\phi)=\hat{F}(\phi,0)$.
\end{dfn} 


We now compute the linearisations of $F$ and $\hat{F}$ at zero, whose kernels will govern infinitesimal deformations of $L$.  
We have in fact already encountered the key operator in \eq{Leq}.

\begin{lem}\label{linlem} For $\phi\in C^2_{\loc}(L)$ and $s\in\R$, 
\begin{align*}
\d F|_0(\phi)&=\mathcal{L}(\phi)=\Delta\phi+\langle\bfx,\nabla\phi\rangle-2\phi\quad\text{and}\\
\d \hat{F}|_{(0,0)}(\phi,s)&=\mathcal{L}(\phi)+s\gamma
\end{align*} where 
$\gamma$ is a smooth function with compact support.
\end{lem}

\begin{proof}
From Lemma \ref{exactlem} we have
 $$\d F|_0(\phi)=\frac{\d}{\d t}(\beta^{t\phi}+\theta^{t\phi})|_{t=0}= \Delta\phi+\langle\bfx,\nabla\phi\rangle-2\phi.$$
Thus $\d_1 \hat{F}|_{(0,0)}(\phi,s)=\mathcal{L}(\phi)$. 

We note that we can find $\alpha_s\in C_{\loc}^\infty(L)$ so that $L^{0,s}=f^{\d\alpha_s}(L)$ and thus, using Lemma \ref{exactlem} again, we obtain
$$ \d_2 \hat{F}|_{(0,0)}(\phi,s)=\frac{\d}{\d t}(\beta^{0,st}+\theta^{0,st})|_{t=0}-s\theta'(0)=s\big(\mathcal{L}(\alpha)-\theta'(0)\big)$$
for some $\alpha\in C_{\loc}^\infty(L)$. We now argue that $\gamma=\mathcal{L}(\alpha)-\theta'(0)$ has compact support and this finishes the proof of the lemma. 

On each connected component of  $T\setminus B_{R_1}$, $G_s$ belongs to $\U(n)$ and we have $\theta^{0,s}=\theta+\theta(s)$ and $\beta^{0,s}-\beta$  a constant $c(s)$ for all $|s|<s_0$. {Next we argue that $c(s)=0$ for all $|s|<s_0$, {which} implies {that} $$\hat F(0,s)=\theta^{0,s}+\beta^{0,s}-\theta(s)=\theta+\theta(s)+\beta-\theta(s)=0$$ outside a compact set, {so}  $\gamma$ {indeed has} compact support.

Recall the diffeomorphism $\pi$ given in \eq{pi.eq}.  From Proposition \ref{tubenbdthm} we see that, by choosing a larger $R_1$ if necessary,  we can find $\chi_s, \psi_s\in C^{\infty}(L_0\setminus B_{R_1})$ with $\chi_0=0,$ $\psi_0=\psi$, respectively, so that, for all $x\in L_0\setminus B_{R_1}$ and $|s|\leq s_0$, we have
$$
f^{0,s}\big(\pi(x)\big)=G_s\circ \Psi_0(x, \d\psi(x))=\Psi_0\big(x,\d\chi_s+\d\psi_s(x)\big).
$$
 Note that $\chi_s$ is a homogeneous quadratic polynomial and the $C^{2}$ norm of $\psi_s$ decays exponentially. Therefore  we have from  Lemma \ref{exactlem} that
$$c(s)=\beta^{0,s}\big(\pi(x)\big)-\beta\big(\pi(x)\big)=-2\chi_s+\langle x, \nabla \chi_s \rangle+o(|x|^{-1}).$$
{Since} $\chi_s$ is a homogeneous quadratic, we have $-2\chi_s+\langle x, \nabla \chi_s \rangle=0$ and thus $c(s)=0$ for all $|s|<s_0.$ 
}
\end{proof}

 
We show that the nonlinear map $F$ is well-defined for $\phi$ in some open ball about zero in $H^{k+2}_*(L)$ for all $k$ large.

\begin{prop}\label{Fdfnprop}
If $k>1+\frac{n}{2}$, there exists $\varepsilon_0>0$ so that $$F:B_{\varepsilon_0}(0)\subset H^{k+2}_*(L)\rightarrow H^{k}(L)\quad\mbox{and}\quad
\hat{F}:B_{\varepsilon_0}(0)\times(-s_0,s_0)\rightarrow H^{k}(L)$$ 
are well-defined.
\end{prop}

\begin{proof}
If $k>1+\frac{n}{2}$ then the Sobolev Embedding Theorem implies that $H^{k+2}\hookrightarrow C^{3,\alpha}$ for some $\alpha\in(0,1)$, so there is an open set containing zero 
in $H^{k+2}$ on which $F$ is defined.  The existence of $\varepsilon_0$ so that $F(\phi)$ and $\hat{F}(\phi,s)$ are defined for $\phi\in B_{\varepsilon_0}(0)$ and $|s|<s_0$ is then immediate as $H^{k+2}_*\subset H^{k+2}$.  

We need to show that $F$ and $\hat F$ take $B_{\varepsilon_0}(0)$ into $H^{k}(L)$. Note that the conditions satisfied by  $G_s\in \U(n)$ in $T\setminus B_{R_1}$ imply that, outside $B_{R_1},$ $\theta^{\phi,s}=\theta^{\phi}+\theta(s)$ and $\beta^{\phi,s}$ differs from $\beta^{\phi}$ by a constant $c(\phi,s)$. Since $\phi$ tends to zero at infinity in $C^3$, we have that $c(s,\phi)=c(s)$ and  we saw in the proof of  Lemma \ref{linlem} that $c(s)=0$. Thus $\beta^{\phi,s}$ is identical to $\beta^{\phi}$ and so $\hat F(\phi,s)$ and $F(\phi)$ are identical functions outside $B_{R_1}$. Hence, we only need to argue that $F$ takes $B_{\varepsilon_0}(0)$ into $H^{k}(L)$.

Notice that $\theta^{\phi}$ depends only on  the tangent space of $L^{\phi}$, and thus on $\nabla\phi$ and $\nabla^2\phi$. Thus we can consider a smooth function of its arguments  $Q_{\theta}(x,y,z)$ so that
$$\theta^{\phi}(x)=\theta(x)+\Delta\phi(x)-\langle H,\overline \nabla\bar \phi\rangle+Q_{\theta}\big(x,\nabla\phi(x),\nabla^2\phi(x)\big).$$
Using  the expression for the  linearization of $\theta^{\phi}$ given in Lemma \ref{exactlem} and arguing as in  \cite[Proposition 2.10]{JoyceSLsing2}, we conclude that $Q_{\theta}$, $\partial_yQ_{\theta}$ and $\partial_zQ_{\theta}$ vanish at $(x,0,0)$. 

From the expression for $\beta^{\phi}$ given in Lemma \ref{exactlem} we see that we can find a smooth function of its arguments  $Q_{\beta}(x,y)$ so that
$$\beta^{\phi}(x)=\beta(x)+\langle \bfx,\nabla\phi(x)\rangle-2\phi(x)+\langle H,\overline \nabla\bar \phi\rangle+Q_{\beta}\big(x,\nabla\phi(x)\big). $$
Using the expression for the linearization of $\beta^{\phi}$ given in Lemma \ref{exactlem} and again arguing as in  \cite[Proposition 2.10]{JoyceSLsing2}, we conclude that $Q_{\beta}$ and $\partial_yQ_{\beta}$ vanish at $(x,0)$.

Therefore, since $F(0)=\theta+\beta=0$, we see that
$$F(\phi)(x)=\mathcal{L}(\phi)(x)+Q\big(x,\nabla\phi(x),\nabla^2\phi(x)\big),$$ where $Q=Q(x,y,z)$ is a smooth function of its arguments 
such that $Q$, $\partial_yQ$ and $\partial_zQ$ all vanish at $(x,0,0)$.  Observe that $Q(x,y,z)$ does not
 directly depend on $\phi(x)$. By Proposition \ref{Lctsprop} it is now enough to show that $Q$ takes $B_{\varepsilon_0}(0)$ into $H^k(L)$.  

Let $\phi\in B_{\varepsilon_0}(0)$ be a smooth function with compact support.  We derive estimates for $\eta$ given by
 $\eta(x)=Q\big(x,\nabla\phi(x),\nabla^2\phi(x)\big)$.  
Since $Q$ and its first derivatives in $y$ and $z$ vanish when $\phi=0$,  
\begin{equation}\label{etaesteq1}
|\eta(x)|\leq C(x)(|\nabla\phi(x)|^2+|\nabla^2\phi(x)|^2)
\end{equation}
 for some non-negative function $C$ on $L$.  Our first objective is to show that $C$ is bounded.
 
 Recall we have a diffeomorphism $\pi:L_0\setminus B_{R_1}\rightarrow L\setminus B$ and set $\phi_0=\pi^*(\phi)$, $F_0(\phi)=\pi^*\big(F(\phi)\big)$. {From Proposition \ref{tubenbdthm} we have 
 $$f^{\phi,0}\big(\pi(x)\big)= \Psi_0\big(x, \d\psi(x)+\d\phi_0(x)\big)$$
 and so on 
 } $L_0\setminus B_{R_1}$ we have $F_0(\phi)=\theta_0^{\phi}+\beta_0^{\phi}$, where  $\theta_0^{\phi}$ and $\beta_0^{\phi}$ are the pull backs of the Lagrangian angle and primitive for $\lambda$ {on} the graph of $\d\phi_0+\d\psi$ over $L_0\setminus B_{R_1}$.  Thus,
 $$F_0(\phi)(x)=\mathcal{L}_0(\phi_0+\psi)(x)+Q_0\big(x,\nabla(\phi_0+\psi)(x),\nabla^2(\phi_0
+\psi)(x)\big),$$
where $\mathcal{L}_0$ is the operator given in \eq{Leq} calculated on $L_0$ and $Q_0$  is a function with the same properties as $Q$.  Since we are working over the planes $L_0$, we compute $$\beta_0^{\phi}=\beta_0+\langle \bfx,\nabla(\phi_0+\psi)\rangle-2(\phi_0+\psi)\quad\text{and}\quad
\theta_0^{\phi}=\theta_0+\sum_{j=1}^n\tan^{-1}(\mu_j)$$ where $\mu_1,\ldots,\mu_n$ are the eigenvalues of $\hess(\phi_0+\psi)$, and $\theta_0$, $\beta_0$ are the Lagrangian angle and primitive for the Liouville form on $L_0$. Thus, because we have chosen $\theta_0+\beta_0=0$, we have
$$Q_0\big(x,\nabla(\phi_0+\psi)(x),\nabla^2(\phi_0+\psi)(x)\big)=\sum_{j=1}^n\tan^{-1}(\mu_j)(x)-\Delta(\phi_0+\psi)(x).$$
  From this explicit formula, we deduce that $Q_0$ and all its derivatives are bounded on $L_0\setminus B_{R_1}$.  Moreover, the 
decay of $|(\partial_x)^aQ_0(x,y,z)|$ is controlled by $|y|^2+|z|^2$. The exponential decay of $\psi$ in Theorem \ref{thm.asympt} implies  $Q$ and $Q_0$ differ by terms with exponential decay and so we have that  $Q$ and all its derivatives are bounded on $L$, and that the $x$ derivatives of $Q$ satisfy
\begin{equation}\label{etaesteq2}
\big|(\partial_x)^aQ\big(x,\nabla\phi(x),\nabla^2\phi(x)\big)\big|\leq C^a(|\nabla\phi(x)|^2+|\nabla^2\phi(x)|^2)
\end{equation}
for some constants $C^a$.  In particular, we can choose $C(x)=C$ independent of $x$
 in \eq{etaesteq1} and we deduce that 
 $$\|\eta\|_{L^2}\leq C\|\phi\|_{C^2}\|\phi\|_{{H}^2}.$$
Since any element of $H^{k+2}_*$ has bounded $C^2$ norm and lies in ${H}^2$, we deduce that $Q$ maps $B_{\varepsilon_0}(0)$ into $L^2$.
  
Now let $j\in\{1,\ldots,k\}$.  Our aim is to show that $\nabla^j\eta$ lies in $L^2$.  By the chain rule,
\begin{equation}\label{nablaetaeq}
|\nabla^j\eta|\leq j!\!\!\!\!\sum_{{\buildrel a,b,c\geq 0\over{\scriptscriptstyle a+b+c\leq j}}}\!\!\!\!|(\partial_x)^a(\partial_y)^b(\partial_z)^cQ|\hspace{-40pt}
\sum_{{\buildrel m_1,\ldots,m_b,n_1,\ldots,n_c\geq 1\over{\scriptscriptstyle a+m_1+\ldots+m_b+n_1+\ldots+n_c=j}}} \!\!
\prod_{i=1}^b|\nabla^{m_i}(\nabla\phi)|\prod_{l=1}^c|\nabla^{n_l}(\nabla^2\phi)|.
\end{equation}
If $j=a$ in the sum in \eq{nablaetaeq}, the products are trivial and we can use \eq{etaesteq2} to show that the corresponding terms
 lie in $L^2$.  Therefore we now assume 
that $j>a$.
Let $q_1,\ldots,q_b,r_1,\ldots,r_c$ be positive constants so that
\begin{equation}\label{qreq}
\sum_{i=1}^b\frac{1}{q_i}+\sum_{l=1}^c\frac{1}{r_l}=1.
\end{equation}
Applying H\"older's inequality to \eq{nablaetaeq}, we see that
\begin{equation}\label{nablaetaeq2}
\begin{split}
\int_L|\nabla^j\eta|^2\d\mathcal{H}^n 
&\leq \!\!\!\!\sum_{{\buildrel a,b,c\geq 0\over{\scriptscriptstyle a+b+c\leq j}}}\!\!\!\!C(a,b,c)\hspace{-30pt}
\sum_{{\buildrel m_1,\ldots,m_b,n_1,\ldots,n_c\geq 1\over{\scriptscriptstyle a+m_1+\ldots+m_b+n_1+\ldots+n_c=j}}} \!\! 
\prod_{i=1}^b\left(\int_L|\nabla^{m_i}(\nabla\phi)|^{2q_i}\d\mathcal{H}^n\right)^{\frac{1}{q_i}} \times\\
&\qquad 
\prod_{l=1}^c\left(\int_L|\nabla^{n_l}(\nabla^2\phi)|^{2r_l}\d\mathcal{H}^n\right)^{\frac{1}{r_l}}
\end{split}
\end{equation}
 for some constants $C(a,b,c)$ determined by $j$ and the derivatives of $Q$, which are bounded.
  
 Given a section $\sigma$ of a vector bundle  with a connection $D$ over $L$ that lies in $H^s\cap L^{\infty}$ we have by \cite[Theorem 3]{Cantor} that
\begin{equation}\label{keyesteq}
\int_L |D^p\sigma|^\frac{2s}{p}\d\mathcal{H}^n\leq C\|\sigma\|_{\infty}^{2\frac{s}{p}-2}\|\sigma\|_{H^s}^2
\end{equation}
for some constant $C$ independent of $\sigma$, whenever $s\geq p$.  (Notice that the results in \cite{Cantor} apply since $L$ is complete, has injectivity radius bounded away from zero and 
bounded sectional curvature.)  
Choosing $q_i=\frac{j-a}{m_i}$ and $r_l=\frac{j-a}{n_l}$, we see that \eq{qreq} holds and we can apply \eq{keyesteq} to deduce that there exists some constant 
$C$, independent of $\phi$, so that
\begin{align}\label{nablaphieq}
\int_L |\nabla^{m_i}(\nabla\phi)|^{2q_i}\d\mathcal{H}^n&\leq C\|\nabla\phi\|_{\infty}^{2q_i-2}\|\phi\|_{H^{j-a+1}}^2 
\quad\text{and}\\
\label{hessphieq}
\int_L|\nabla^{n_l}(\nabla^2\phi)|^{2r_l}\d\mathcal{H}^n&\leq C\|\nabla^2\phi\|_{\infty}^{2r_l-2}\|\phi\|_{H^{j-a+2}}^2. 
\end{align}
Therefore, 
substituting \eq{nablaphieq} and \eq{hessphieq} into \eq{nablaetaeq2} we see that there {exists} a 
constant $C(j,Q,\|\phi\|_{C^2})$ 
so that
$$\int_L|\nabla^j\eta|^2\d\mathcal{H}^n\leq C(j,Q,\|\phi\|_{C^2})\|\phi\|_{H^{j+2}}^2.$$
 Since this holds for all smooth compactly supported $\phi\in B_{\varepsilon}(0)$ we see that 
$\eta\in H^j$ whenever $\phi\in H^{j+2}\cap C^2$ for $j=1,\ldots,k$.  The result for $F$ follows.

\end{proof}

We can now prove the following local uniqueness result.  

\begin{thm}\label{defthm1}    
Let $k>1+\frac{n}{2}$.  There exist $0<{\varepsilon_1}\leq\varepsilon_0$ and $0<s_1\leq s_0$ so that
for each $|s|<s_1$ there exists a unique $\phi(s)\in B_{\varepsilon_1}(0)\subseteq H^{k+2}_*(L)$ so that $\hat{F}\big(\phi(s),s\big)=0$, where $s\mapsto\phi(s)$ is continuous.

Moreover, if $(\phi,s)\in B_{{\varepsilon_1}}(0)\times(-s_1,s_1)$ then $L^{\phi,s}$ is a self-expander with $H=\bfx^\bot$ 
if and only if $\phi=\phi(s)$.
\end{thm}

\begin{proof}
By Proposition \ref{Fdfnprop}, $\hat{F}:B_{{\varepsilon}_0}(0)\times(-s_0,s_0)\rightarrow H^k(L)$ is well-defined and $L^{\phi,s}$ is a self-expander if and only if $\hat{F}(\phi,s)$ is constant.  However, 
if $\phi\in B_{\varepsilon_0}(0)$ then $\hat{F}(\phi,s)\in C^1(L)\cap L^2(L)$ by the Sobolev Embedding Theorem and so $|\hat{F}(\phi,s)(x)|\rightarrow 0$ as $|x|\rightarrow \infty$.  Hence 
$L^{\phi,s}$ is a self-expander if and only if $\hat{F}(\phi,s)=0$.   
 
By Lemma \ref{linlem}, $\d\hat{F}|_{(0,0)}(\phi,s)=\mathcal{L}(\phi)+s\gamma$ and $\gamma$ has compact support.  Corollary \ref{Lisocor} implies that $\d_1\hat{F}|_{(0,0)}=\mathcal{L}:H^{k+2}_*(L)\rightarrow H^k(L)$ is an isomorphism.  Thus $\d\hat{F}|_{(0,0)}:H^{k+2}_*(L)\times\R\rightarrow H^k(L)$ is 
  surjective.  Moreover, there exists unique $\Gamma\in H^{k+2}_*(L)$ such that $\mathcal{L}(\Gamma)=-\gamma$, so $\d\hat{F}|_{(0,0)}$ has a 1-dimensional kernel.  

Applying the Implicit Function Theorem for Banach spaces \cite[Chapter XIV, Theorem 2.1]{Lang}, we see that  there 
exist ${\varepsilon}_1\leq{\varepsilon}_0$, $s_1\leq s_0$ and a unique continuous map $s\mapsto\phi(s)$ so that 
$$\hat{F}^{-1}(0)\cap \big(B_{{\varepsilon_1}}(0)\times(-s_1,s_1)\big)=\big\{\big(\phi(s),s\big)\,:\,|s|<s_1\big\}.$$ 
 The result follows.   
\end{proof}

We now finish the proof of Theorem \ref{local.uniqueness}.

From Theorem \ref{defthm1} we obtain, for all $|s|<s_1$, the existence of a zero-Maslov class self-expander $L^s$ asymptotic to $L^s_0=P_1(s)+P_2(s)$.  The family $(L_s)_{|s|<s_1}$ is continuous in $C^{2,\alpha}$.

To show uniqueness, apply Theorem \ref{thm.asympt} with $k=3$ and $$K=\big\{\big(P_1(s),P_2(s)\big)\big\}_{|s|\leq s_1}\subset G_n$$ to obtain the existence of $\varepsilon$ and $R_0$ so that if $N^s$ is a self-expander asymptotic to $L^s_0$ which is  $\varepsilon$-close in $C^2$ to $L$ in $B_{R_0}$, then $N^s=L^{\phi,s}$ for some $\phi\in B_{\varepsilon_1}$. Theorem   \ref{defthm1} implies $N^s$ is unique and equal to $L^s$.
{\hfill$\square$}

\section{Compactness Theorem in $\C^2$}\label{compactness.section}

We now restrict to the situation where the self-expander is asymptotic to transverse planes in $\C^2$.  The reason is that it is 
only in $\C^2$ where a Lagrangian cone with density strictly less than $2$ must be a plane.  For $n>2$, the Harvey--Lawson
 $\U(1)^{n-1}$-invariant special Lagrangian cone in $\C^n$ has density strictly between $1$ and $2$.

Consider $M=G_L(2,\C^2)\times G_L(2,\C^2)$, where  $G_L(2,\C^2)$ denotes the set of all multiplicity one 
Lagrangian planes in $\C^2$.  Define
$$SL=\{(P_1,P_2)\in M\,|\, P_1+P_2\mbox{ or }P_1-P_2\mbox{ is area-minimizing}\}$$
 and 
\begin{equation}\label{lambdaeq} 
\Lambda=\{(P_1,P_2)\in M\,|\, P_1\cap P_2=\{0\}\}\setminus  SL.
\end{equation}
 Since $(P_1,P_2)\in M$ lies in $SL$ if and only if the sum of the angles between $P_1$ and $P_2$ is an integer multiple of $\pi$, we see that 
 $\Lambda$ is an open subset of $M$.

The aim of this section is to prove the following compactness result.

\begin{thm}\label{compactness}
Pick a compact set ${K}\subset   \Lambda$. The set 
\begin{align*}
\mathcal S(K)=\{L\subseteq\C^2\,|\,& L\mbox{ is a 
zero-Maslov class Lagrangian self-expander }\\
&
\mbox{
which is asymptotic to $P_1+P_2$ where }(P_1,P_2)\in {K}\}
\end{align*}
is compact in the $C^{2,\alpha}$ topology.
\end{thm}
\begin{proof}
Let $(L^i)_{i\in\N}$ be a sequence of self-expanders in $\mathcal S(K)$ asymptotic to $L^i_0=P^i_1+P^i_2$. 
Setting $L^i_t=\sqrt{2t}L^i$, we thus have a sequence
$(L^i_t)_{t\geq 0}$ 
of solutions to Lagrangian mean curvature flow which are smooth for all $t>0$. From Lemma \ref{embedlem} we have uniform area
 bounds for
 $(L^i_t)_{t\geq 0}$ and so \cite[Theorem 7.1]{ilmanen1} implies that we can consider a subsequence which converges weakly to an integral Brakke motion $(L_t)_{t\geq 0}$. 

 It also follows from \cite[Theorem 7.1]{ilmanen1} 
that, for almost all $t>0$, $L^i_t$ admits a subsequence which converges to $L_t$ as an integral varifold and so $2tH=\x^{\bot}$ on 
$L_t$. Furthermore, Radon measure convergence implies that $L_t=\sqrt{2t}L_{1/2}$ for all $t>0$. In particular, $L_{1/2}$ is an 
integral varifold with $H=\x^{\bot}$ and we denote it simply by $L$. 

 Compactness of ${K}$ implies that, after passing to another subsequence, $(P^i_1,P^i_2)$ converges to $(P_1,P_2)\in {K}$.  Our objective 
is to show that $L\in\mathcal{S}(K)$.

We first show that $L$  is asymptotic to $L_0=P_1+P_2$.
 \begin{lemm}\label{lemm.se.asympt} There is $R_0>0$ and $\psi\in C^{\infty}(L_0\setminus B_{R_0})$ so that
$$L\setminus B_{2R_0}\subset\{x+J\overline \nabla\psi(x)\,|\, x\in L_0
\setminus B_{R_0}\}  \subset L\setminus B_{R_0/2}$$
and, for some $b>0$, 
$$\|\psi\|_{C^{3,\alpha}(L_0\setminus B_R)}=O(e^{-bR^2})
\mbox{ as }R\rightarrow\infty.$$ 
Moreover, $L^i\setminus B_{R_0}$ converges to $L\setminus B_{R_0}$ in $C^{2,\alpha}$ as $i\rightarrow\infty$.
\end{lemm}
 \begin{proof}  The lemma follows from Theorem \ref{thm.asympt}.
  \end{proof}

We can now deduce that $L$ is asymptotic to $L_0$.

\begin{lemm}\label{time.zero} As Radon measures, $L_t\rightarrow L_0$ as $t\rightarrow 0$. 
\end{lemm}
\begin{proof}
Given $\varepsilon>0$ small we obtain from Lemma \ref{lemm.se.asympt} that, for all $t$ sufficiently small,
$$L_{t}\setminus B_{2\varepsilon}\subset\{x+JY_t(x)\,|\, x\in L_0\setminus B_{\varepsilon}\}  \subset L_t\setminus B_{\varepsilon/2},$$
where the $C^{2,\alpha}$ norm of the vector field $Y_t$ tends to zero as $t\rightarrow 0$. Thus, 
$$\lim_{t\to 0^+}\int_{L_t}\phi\, \d\H^2 =\int_{L_0}\phi\, \d\H^2$$
for all $\phi\in C^{\infty}_0(\C^2)$.
 \end{proof}

The next proposition is one of the key steps to ensure that $L$ is smooth.
 
\begin{prop}\label{stationary}
$L$ is not a stationary varifold. 
\end{prop}
\begin{proof}
Assume $L$ is stationary.  Then $L$ needs to be a cone because $\x^{\bot}=H=0$.
Thus $L_t=\sqrt{2t}L$ has $H=0$ for all $t>0$ and we obtain from Lemma \ref{time.zero} that $L=P_1+P_2$. The goal for the rest  of this proof is to show that $L$ must be area-minimizing and this gives us a contradiction because $(P_1,P_2)\notin SL$.

Since $L^i$ is a self-expander we have (from varifold convergence) that for every $r>0$
$$\lim_{i\to\infty}\int_{L^i\cap B_r}|\x^{\bot}|^2\d\H^2=\int_{L\cap B_r}|\x^{\bot}|^2\d\H^2=0$$
and thus, for all $r>0$,  
\begin{equation}\label{zero}
\lim_{i\to\infty}\int_{L^i\cap B_r} \big(|H|^2+|\x^{\bot}|^2\big)\d\H^2=\lim_{i\to\infty}\int_{L^i\cap B_r} \big(2|\x^{\bot}|^2\big)\d\H^2=0.
\end{equation}
\begin{lemm}\label{props.poincare}The following properties hold:
\begin{itemize}
\item[(i)]There is $d_0>0$ 
so that for every $R>0$, every $i$ sufficiently large, and every open subset $A$ of $L^i\cap B_{4R}$ with rectifiable boundary we have
\begin{equation*}\label{isoperimetric}
\left(\H^2(A)\right)^{1/2}\leq d_0\H^{1}(\partial A).
\end{equation*}
\item[(ii)] There is $R_1>0$ so that for all $R>R_1$ and all $i$ sufficiently large
 $$L^i\cap B_{2R} \mbox{ is connected  and }\partial(L^i\cap B_{3R})\subset \partial B_{3R}.$$ 
\item[(iii)]There is $c>0$ so that  for all  $i$ sufficiently large we have
$$\sup_{L^i}|\theta^i|= \sup_{L^i}|\beta^i|\leq c.$$
\end{itemize}
\end{lemm} 
\begin{proof}
We first prove (i). From {the}  Michael--Simon Sobolev inequality
(see \cite[Theorem 18.6]{Leon})
$$
\left(\H^2(A)\right)^{1/2}\leq c_0\int_A |H| +c_0\H^{1}(\partial A)
$$
for some universal 
constant $c_0>0$. In this case we have
$$
\left(\H^2(A)\right)^{1/2}\leq c_0\left(\H^2(A)\right)^{1/2} \left(\int_A |H|^2\right)^{1/2}+c_0\H^{1}(\partial A)
$$
and so we get the desired claim because for all $i$ sufficiently large we have (due to \eqref{zero}) 
$$c_0^2 \int_{L^i\cap B_{4R}} |H|^2\leq \frac{1}{4}.$$

Property (ii) follows from  Lemma \ref{lemm.se.asympt}.

Finally, we prove property (iii). Given $y_i\in L^i$, denote by $\hat{B}_{r}(y_i)$ the intrinsic ball  in $L^i$ of radius $r$ {and set $\psi_i(r)=\H^2(\hat{B}_{r}(y_i)).$
From (i) we see that for almost all $r$
$$\big(\psi_i(r)\big)^{1/2}\leq d_0\H^1\big(\partial \hat{B}_{r}(y_i)\big) =d_0\psi'_i(r).$$
Integrating the above inequality implies} the existence of  $d_1>0$, depending only on $d_0$, so that  for all $R>0$
\begin{equation}\label{lower.bound}
\H^{2}\big(\hat{B}_{r}(y_i)\big)\geq d_1r^2\mbox{ for all }y_i\in B_{3R}\cap L^i \mbox{ and } r<R.
\end{equation}
Choose $\beta^i$, the primitive for the Liouville form $\lambda|_{L_i}$, so that $\beta^i+\theta^i=0$ ($L^i$ is a self-expander).
Combining the uniform area bounds given in Lemma \ref{embedlem} with \eqref{lower.bound}, we have that the intrinsic diameter of $L^i\cap B_{R}$ is uniformly bounded for all  $i$ sufficiently large. Hence, if $x,y\in L^i\cap B_{R}$ and $\gamma$ is a path in $L^i\cap B_{R}$ connecting $x$ to $y$, we have
$$\beta^i(x)-\beta^i(y)=\int_{\gamma}\lambda\leq R\,\mbox{length}(\gamma).$$
Thus the oscillation  of $\beta^i$ in $L^i\cap B_{R}$ is uniformly bounded. The angle $\theta^i=-\beta^i$ can always be chosen so that its range in $L^i\cap B_{R}$ intersects the interval $[0,2\pi]$ and so we obtain that $\theta^i=-\beta^i$ is uniformly bounded in $B_R$. 

From Lemma \ref{lemm.se.asympt} we know that $\theta^i=-\beta^i$ are uniformly bounded outside a large ball and thus 
are uniformly bounded on $L^i$.
\end{proof}

We can now finish the proof of Proposition \ref{stationary}. Recall that $L=P_1+P_2$ in the varifold sense and, if necessary, we can change the orientation of one of the planes so that the  identity also holds in the current sense. We want to show that $L=P_1+P_2$ is area-minimizing.

We know that for all $R>0$,
$$\lim_{i\to\infty}\int_{L^i\cap B_R}|H|^2+|\x^{\bot}|^2\d\H^2=0$$
and $|\x^{\bot}|=|\nabla \beta^i|$ is uniformly bounded in $B_R$. From Lemma \ref{props.poincare} we have that all conditions necessary to apply   \cite[Proposition A.1]{neves} are met and so we conclude the  existence of a constant $\bar \beta$  and $R_2$ such that, for all $\phi\in {C}^{\infty}_0(\C^2)$, 
        \begin{equation*}\label{poincare}
        \lim_{i\to\infty}\int_{L^i\cap B_{R_2}}(\beta^i-\bar\beta)^2\phi \,\d\H^2=0.
        \end{equation*}
        Hence
        $$\lim_{i\to\infty}\int_{L^i\cap B_{R_2}}(\theta^i+\bar\beta)^2\phi\, \d\H^2=0$$
for all $\phi\in {C}^{\infty}_0(\C^2)$.  We deduce, from \cite[Proposition 5.1]{neves}, that $L$ has constant Lagrangian angle  {$-\bar \beta$} and   is thus  area-minimizing, providing our required contradiction.
\end{proof}

Using the fact that $L$ is not stationary, we now show that $L$ satisfies the conditions of White's Regularity Theorem. 
It is in this lemma that we use the fact that $n=2$ in a crucial way. Recall the definition of Gaussian density in \eqref{gaussian}.

\begin{lemm}\label{small}
Given $\varepsilon_0>0$ small, there is $\delta>0$ so that
        $$\Theta_t(y,l)\leq 1+\frac{\varepsilon_0}{2}\quad\mbox{ for every } l\leq \delta t,\,\,y\in \C^2\mbox{ and }t>0.$$
\end{lemm}
\begin{rmk}
We briefly sketch the idea. The first step is to find $\delta$ so that $\Theta_{1/2}(y,l)<2$ for all $y\in\C^n$ and $l\leq \delta$. This follows because the monotonicity formula implies that  $\Theta_{1/2}(y,l)\leq\Theta_{0}(y,l+1/2)\leq 2$ with equality only if $L$ is a self-shrinker centered at the origin. In the latter case, because $L$ is a self-expander, we obtain that $L$ must be stationary, which contradicts Proposition \ref{stationary}. Thus, the strict inequality holds as claimed.

The second step is to show that if $\Theta_{1/2}(y_i,\delta_i)\geq 1+\frac{\varepsilon_0}{2}$ for some sequence $\delta_i$ tending to zero, then we can blow-up $L$ and obtain a stationary Lagrangian varifold $\tilde L$ which is not  a plane. Then we blow-down $\tilde L$ to obtain a stationary Lagrangian cone $C$ which must have Gaussian density  at the origin bigger than $1+\frac{\varepsilon_0}{2}$. Since we are in $\C^2$, this forces the Gaussian density at the origin to be at least two, which we then show contradicts the first step. 
\end{rmk}
\begin{proof}
        It suffices to prove the lemma for $t=1/2$ because $L_t=\sqrt{2t}L.$  
        
        In what follows we will constantly use the fact that, because $P_1$ intersects $P_2$ transversely,
        $$\int_{P_1+P_2}\Phi(y,l) \d\H^2< 2\quad\mbox{for all }l>0\mbox{ and }y\neq 0,$$ 
        with equality holding  if $y=0$.        
        \medskip
        
       \noindent{\bf First step:} We start by arguing the existence of $c_1>0$  such
 that, for every $l\leq 2$ and $y\in \C^2$,
        \begin{equation}\label{density}
                \int_{L}\Phi(y,l) \,\d\H^2\leq 2-c_1^{-1}.
        \end{equation}
From the monotonicity formula for Brakke flows \cite[Lemma 7]{ilmanen},
\begin{align}\label{huisken}
        \int_{L}&\Phi(y,l) \d\H^2\\
&+\int_{0}^{1/2}\int_{L_t}\left|H+\frac{(x-y)^{\bot}}{2(l+1/2-t)}\right|^2\Phi(y,l+1/2-t) \d\H^2 \d t\nonumber\\
        &= \int_{P_1+P_2}\Phi(y,l+1/2) \d\H^2\leq 2.\nonumber
\end{align}
 Suppose there is a sequence $y_i$ and $l_i$ with $0\leq l_i\leq 2$ such that

$$\int_{L}\Phi(y_i,l_i) \d\H^2\geq 2-\frac 1 i.$$
Then, by \eqref{huisken},
$$ \int_{P_1+P_2}\Phi(y_i,l_i+1/2) \d\H^2\geq 2-\frac 1 i$$
and so  $y_i$ must converge to zero. 

Assuming without loss of generality that $l_i$ converges to $\bar l$, we have again from \eqref{huisken} that
 \begin{align*}
\int_{0}^{1/2}&\int_{L_t}\left|H+\frac{x^{\bot}}{2(\bar l+1/2-t)}\right|^2\Phi(0,\bar l+1/2-t) \d\H^2 \d t\\
&= \lim_{i\to\infty}\int_{0}^{1/2}\int_{L_t}\left|H+\frac{(x-y_i)^{\bot}}{2(l_i+1/2-t)}\right|^2\Phi(y_i,l_i+1/2-t) \d\H^2 \d t\\
&\leq 2-\lim_{i\to\infty}\int_{L}\Phi(y_i,l_i) \d\H^2 =0
\end{align*}
and thus
$$ H+\frac{\bfx^{\bot}}{2(\bar l+1/2-t)}=0\mbox{ on }L_t\mbox{ for almost all } t\in[0,1/2].$$
Combining this with the fact that 
$H=\frac{\bfx^{\bot}}{2t}$ on $L_t$
we obtain that $L$ must be stationary, which contradicts Lemma \ref{stationary}. Thus \eqref{density} must hold.

         \medskip
        
        \noindent{\bf Second step:} To finish the proof we argue again by contradiction and assume the existence of sequences  $(y_j)_{j\in\N}$ in $\C^2$ and $(\delta_j)_{j\in\N}$ converging to zero so that
        \begin{equation}\label{bound.below.L}
        \Theta_{1/2}(y_j,\delta_j)\geq 1+\frac{\varepsilon_0}{2}.
        \end{equation}
        From the monotonicity formula for Brakke flows \cite[Lemma 7]{ilmanen} we have
        \begin{equation*}
        \int_{P_1+P_2}\Phi(y_j,\delta_j+1/2) \d\H^2\geq \Theta_{1/2}(y_j,\delta_j)\geq 1+\frac{\varepsilon_0}{2}.
        \end{equation*}
        Note that  the sequence  $(|y_j|)_{j\in \N}$ is bounded by a positive constant $M_0$, because otherwise we could find a subsequence so that
        $$\lim_{j\to\infty} \int_{P_1+P_2}\Phi(y_j,\delta_j+1/2) \leq 1.$$ 

        Consider the sequence of blow-ups
         \begin{equation*}
        \tilde L^{j,i}_s=\delta_j^{-1/2}\left(L^{i}_{1/2+s\delta_j}-y_j\right)\, \mbox{ with } s\geq 0.
        \end{equation*}
        A standard diagonalization argument allows us to consider a subsequence $\tilde L^{j}_s=\tilde L^{j,i(j)}_s$ 
        such that, for all $1\leq  l \leq j$,
        \begin{equation}\label{sequence.bound}
        -\frac{1}{j}\leq\int_{\tilde L^j_0}\Phi(0,l) \d\H^2- \int_{L} \Phi(y_j, l\delta_j)  \d\H^2\leq \frac{1}{j}.
        \end{equation}
        Thus, for every $r>0$,
        \begin{align*}
                \int_0^1\int_{\tilde L^j_s\cap B_r}|H|^2\d\H^2 \d s&= \delta_j^{-1}\int_{1/2}^{1/2+\delta_j}\int_{L^{{i(j)}}_t\cap B_{\sqrt \delta_j r}(y_j)}|H|^2\d\H^2 \d t\\
                &= \delta_j^{-1}\int_{1/2}^{1/2+\delta_j}\int_{L^{{i(j)}}_t\cap B_{\sqrt \delta_j r}(y_j)}\left|\frac{\bfx^{\bot}}{2t}\right|^2\d\H^2 \d t \leq c_2 \delta_j
        \end{align*}
        where $c_2$ depends on $r$ and $M_0$.
        Therefore
        $$\lim_{j\to\infty}\int_0^1\int_{\tilde L^j_s\cap B_r}|H|^2\d\H^2 \d s=0$$
        and so $(\tilde L^j_s)_{0\leq s\leq 1}$ converges to an integral Brakke flow $(\tilde L_s)_{0\leq s\leq 1}$ with $\tilde L_s=\tilde L$ for all $s$ and $\tilde L$ a stationary varifold. From \cite[Proposition 5.1]{neves} we conclude that $\tilde L$ is a union of special Lagrangian currents (the Lagrangian angle is uniformly bounded by Lemma \ref{props.poincare} (iii)).

 From \eqref{bound.below.L} and \eqref{sequence.bound} we have
        $$\int_{\tilde L} \Phi(0,1) \d\H^2\geq 1+\frac{\varepsilon_0}{2}$$ and so $\tilde L$ cannot be a  plane with multiplicity one. The blow-down $$C=\lim_{i\to 0} \varepsilon_i \tilde L, \quad\mbox{where}\quad\varepsilon_i\to 0,$$ is a union of Lagrangian planes (as these are the only special Lagrangian cones in $\C^2$) and so
        \begin{multline*}
        \lim_{i\to\infty}\int_{\tilde L}\Phi(0,\varepsilon_i^{-2}) \d\H^2=\lim_{i\to\infty}\int_{\varepsilon_i\tilde{L}}\Phi(0,1) \d\H^2 
        =\int_{C}\Phi(0,1) \d\H^2 \geq2.
        \end{multline*}
        From \eqref{sequence.bound} this implies that one can find $l$  such that for every $j$ sufficiently large we have
        $$
                2-\frac{1}{2c_1}\leq\int_{\tilde L^j_0}\Phi(0,l) \d\H^2\leq \int_{L}\Phi(y_j,l \delta_j) \d\H^2+\frac{1}{j}.  
        $$
        This contradicts \eqref{density}.
\end{proof}

We may now complete the proof of Theorem \ref{compactness}. From Lemmas \ref{lemm.se.asympt} and \ref{small} we have that, for all $i$ sufficiently large,
$$\Theta^i_t(y,l)\leq 1+{\varepsilon_0}\quad\mbox{ for every } l\leq \delta t,\,\,y\in \C^2,\mbox{ and }t>0,$$
where $\Theta^i_{t}(y,l)$ is the Gaussian density \eqref{gaussian} of $L^i_t$.   

White's Regularity Theorem \cite{white} implies uniform $C^{2,\alpha}$ bounds for $L^i_{1/2}$ and so  $L=L_{1/2}$ is a smooth multiplicity one self-expander asymptotic to $P_1+ P_2$ with $(P_1,P_2)\in{K}$ and $L^i_{1/2}$ converges to $L$ in  $C^{2,\alpha}$.  Hence, $L\in\mathcal{S}(K)$  as we wanted to show.
\end{proof}

\section{Uniqueness Theorem in $\C^2$}\label{uniqueness.section}

We first prove the uniqueness  for self-expanders which are asymptotic to planes $P_1+P_2$, where $P_1$ and $P_2$ share the same $S^1$-symmetry.
\subsection*{Equivariant case}
We say a Lagrangian surface $N\subset \C^2$ is equivariant if there is a curve $\gamma:\R\rightarrow\C$ or $\gamma:[0,\infty)\rightarrow \C$ so that
$$N=\{(\gamma(s)\cos\alpha, \gamma(s)\sin\alpha)\,|\,s\in \R, \alpha\in[0,2\pi]\}\subset \C^2.$$
Consider the ambient function $\mu=x_1y_2-y_2x_1$.  The relevance of this function is that {an} {embedded Lagrangian} $N$ is equivariant if and only if $N\subset \mu^{-1}(0)$ (see \cite[Lemma 7.1]{neves3} for instance).
 
 Studying the o.d.e. arising from $H=\bfx^{\perp}$, Anciaux \cite{Anciaux} showed that given two equivariant planes $P_1, P_2$, there is a unique  equivariant  Lagrangian self-expander  $L$  asymptotic to  $L_0=P_1+P_2$.
 
 \begin{lemm}\label{anciaux} Suppose that $L$ is a zero-Maslov class Lagrangian self-expander  asymptotic to $L_0=P_1+P_2$, where $P_1,P_2$ are equivariant planes. Then $L$ is equivariant. In particular, it is unique.
 \end{lemm}
 \begin{proof}
From \cite[Lemma 3.3]{neves2} we know that along $L_t=\sqrt{2t}L$,
$$ \frac{\d}{\d t}\mu^2=\Delta\mu^2-2|\nabla \mu|^2.$$
Using the evolution equation above in Huisken's monotonicity formula  we have that, for $t<1$,
$$\frac{\d}{\d t}\int_{L_t} \mu^2\Phi(0,1-t) \d\H^2\leq 0.$$
From Theorem \ref{thm.asympt} we see that
$$\lim_{t\to 0}\int_{L_t} \mu^2\Phi(0,1-t) \d\H^2= \int_{P_1+P_2} \mu^2\Phi(0,1)\d\H^2=0$$
because $P_1, P_2$ are equivariant. Thus 
$$\int_{L_t} \mu^2\Phi(0,1-t) \d\H^2=0$$
for all $t<1$ and this implies the desired result.
 \end{proof}

\subsection*{General case}

 Consider the set $\Lambda \subset G_L(2,\C^2)\times G_L(2,\C^2)$ defined in \eqref{lambdaeq}
 $$\Lambda=\{(P_1,P_2)\,|\, P_1\cap P_2=\{0\}\}\setminus  SL,$$
and let 
$$E=\Lambda\cap \{(P_1,P_2)\,|\,P_1\subset \mu^{-1}(0),P_2\subset \mu^{-1}(0)\}.$$
 
 \begin{thm} Given $(P_1,P_2)\in \Lambda$ there is a unique zero-Maslov class self expander $L$ which is asymptotic to $L_0=P_1+P_2$.
 \end{thm}
 \begin{rmk} The existence of such self-expanders was proven by Joyce--Lee--Tsui in \cite{jlt}.  Explicit formulae for the self-expanders are given in \cite[Theorem C]{jlt}.
 \end{rmk}
 \begin{proof}
 
 Choose  a basis of $\C^2$ so that $P_1$ is the real plane and $$P_2=\{(e^{i\theta_1}x,e^{i\theta_2}y)\,|\, x,y\in \R\}.$$
 Set
 $$P_2(s)=\{(e^{i\theta_1-s(\theta_1-\theta_2)/2}x,e^{i\theta_2+s(\theta_1-\theta_2)/2}y)\,|\, x,y\in \R\}.$$
 The key properties of $\big(P_1,P_2(s)\big)$ are that
 \begin{itemize}
 \item $\big(P_1,P_2(s)\big)\in \Lambda$ for all $0\leq s\leq 1$ and {the} {Lagrangian angle of  $P_2(s)$} is constant;
 \item  $P_1\subset \mu^{-1}(0)$ and $P_2(1)\subset \mu^{-1}(0)$.
 \end{itemize}
Consider the compact subset of $\Lambda$ given by  $K=\big\{\big(P_1,P_2(s)\big)\big\}_{0\leq s\leq 1}$ and  recall 
\begin{align*}
\mathcal S(K)=\{L\subseteq\C^2\,|\,& L\mbox{ is a zero-Maslov class Lagrangian self-expander}\\
&\mbox{
which is asymptotic to $P_1+P_2$ where }(P_1,P_2)\in {K}\}.
\end{align*}
Consider the obvious projection map
 $\pi:\mathcal{S}(K)\rightarrow[0,1]$.

From \cite[Theorem C]{jlt} we know that $\pi$ is surjective.
  By Theorem \ref{local.uniqueness}, one may choose a suitable topology on $\mathcal{S}(K)$
 so that $\pi$ is a local diffeomorphism.  By Theorem \ref{compactness}, {$\mathcal{S}(K)$} is also compact with respect to this
 topology.  Therefore $\pi$ is a covering map.  
However, by Lemma \ref{anciaux}, we have that  $\pi^{-1}(1)$ consists of a single element and so $\pi$ is a
 diffeomorphism. In particular, $\pi^{-1}(0)$ consists of a single element.
 \end{proof}

\end{document}